\documentclass[10pt,fleqn]{amsproc}

\usepackage[utf8]{inputenc}
\usepackage[T1]{fontenc}
\usepackage{lmodern}

\usepackage{amsmath,amssymb,amsthm}
\usepackage{mathtools}
\usepackage{mathrsfs}
\usepackage{graphicx}

\usepackage{tikz}
\usepackage{tikz-cd}

\usetikzlibrary{arrows.meta,calc,decorations.markings,intersections}
\usepackage{xcolor}
\usepackage{hyperref}
\hypersetup{colorlinks=true,linkcolor=magenta,citecolor=blue,urlcolor=blue}

\usepackage{enumitem}

\newcommand{\textupbf}[1]{\textbf{\textup{#1}}}

\allowdisplaybreaks

\SetEnumitemKey{auto}{label=\textupbf{(\arabic*)},ref=(\arabic*),leftmargin=*}
\SetEnumitemKey{autoarabic}{label=\textupbf{(\arabic*)},ref=(\arabic*),leftmargin=*}
\SetEnumitemKey{autoalpha}{label=\textupbf{(\alph*)},ref=(\alph*),leftmargin=*}
\SetEnumitemKey{autoroman}{label=\textupbf{(\roman*)},ref=(\roman*),leftmargin=*}
\SetEnumitemKey{autoAlpha}{label=\textupbf{(\Alph*)},ref=(\Alph*),leftmargin=*}
\SetEnumitemKey{autoRoman}{label=\textupbf{(\Roman*)},ref=(\Roman*),leftmargin=*}

\newcommand{\rmop}[1]{\operatorname{#1}}

\numberwithin{equation}{section}
\numberwithin{figure}{section}

\newtagform{bold}{\bfseries(}{\bfseries)}
\usetagform{bold}

\newcommand{\purplebfeq}[1]{\textcolor{purple!80!black}{\bfseries(#1)}}
\newtagform{bfpurple}[\purplebfeq]{}{}

\newcommand{\frakP}{\mathfrak{P}}

\newcommand{\bbF}{\mathbb{F}}

\newcommand{\bbQ}{\mathbb{Q}}

\newcommand{\bbZ}{\mathbb{Z}}

\newcommand{\calO}{\mathcal{O}}

\newcommand{\scrK}{\mathscr{K}}

\newcommand{\scrO}{\mathscr{O}}

\newcommand{\scrS}{\mathscr{S}}
\newcommand{\scrT}{\mathscr{T}}

\newcommand{\rmH}{\mathrm{H}}

\title[Minimal degree isogenies between superspecial abelian surfaces]{A uniform bound for the minimal degree of an isogeny between principally polarized superspecial abelian surfaces}
\author{Lam L. Pham}
\email{lam.pham@ulb.be}
\address{Department of Computer Science, Universit\'e Libre de Bruxelles}
\thanks{L.L. Pham (ULB) is supported by the FRS-FNRS through grant number T.0238.25}

\usepackage[backend=biber,style=alphabetic,maxbibnames=99,maxalphanames=3,
  giveninits=true,doi=true,isbn=false,url=true]{biblatex}
\providecommand{\thmabovespace}{6pt}
\providecommand{\thmbelowspace}{6pt}

\newtheoremstyle{standard}%
  {\thmabovespace}{\thmbelowspace}
  {} 
  {0pt}{\bfseries\color{blue!80!black}}{.}{1em}
  {\thmname{#1}\thmnumber{~#2}\thmnote{~({\normalfont{#3}})}}

\newtheoremstyle{standardblue}%
  {\thmabovespace}{\thmbelowspace}
  {} 
  {0pt}{\bfseries\color{blue!80!black}}{.}{1em}
  {\thmname{#1}\thmnumber{~#2}\thmnote{~({\normalfont{#3}})}}
  
\newtheoremstyle{standarditred}%
  {\thmabovespace}{\thmbelowspace}
  {\itshape} 
  {0pt}{\bfseries\color{red!80!black}}{.}{1em}
  {\thmname{#1}\thmnumber{~#2}\thmnote{~({\normalfont{#3}})}}

\newtheoremstyle{standardit}%
  {\thmabovespace}{\thmbelowspace}
  {\itshape} 
  {0pt}{\bfseries}{.}{1em}
  {\thmname{#1}\thmnumber{~#2}\thmnote{~({\normalfont{#3}})}}

\theoremstyle{standardit}

\newtheorem{theorem}{Theorem}

\newtheorem{proposition}{Proposition}[section]
\newtheorem{lemma}[proposition]{Lemma}

\newtheorem*{claim*}{Claim}

\theoremstyle{standard}

\newtheorem{remark}[proposition]{Remark}

\DeclareMathOperator{\Herm}{Herm}
\DeclareMathOperator{\Hdet}{det_{\rmH}}
\DeclareMathOperator{\nrd}{nrd}

\DeclareMathOperator{\Trd}{Trd}
\DeclareMathOperator{\Tr}{Tr}
\DeclareMathOperator{\trd}{trd}
\DeclareMathOperator{\tr}{tr}
\DeclareMathOperator{\Hom}{Hom}
\DeclareMathOperator{\End}{End}

\newcommand{\medsim}{\ensuremath\mathrel{\scalebox{1.05}{\ensuremath\sim}}}
\newcommand{\stosim}[1][]{\stackrel{\medsim}{\smash{\xrightarrow[#1]{\hspace{0.5em}}}\rule{0pt}{0.2ex}}}

\newcommand{\m}[2][p]{\begin{#1matrix}#2\end{#1matrix}}

\newcommand{\tm}[2][p]{\begin{#1tinymatrix}#2\end{#1tinymatrix}}

\begin{document}

\begin{abstract}
Let \(p>3\)  be a prime number.
If \(A\) and \(B\) are two principally polarized superspecial abelian surfaces over \(\bbF_{p^{2}}\) with \(p^{2}\)-Frobenius \([-p]\), we prove that there exists a separable polarized isogeny between them with multiplier at most \(p/\sqrt{2}\).
The multiplier bound is uniform in the pair and asymptotically optimal up to a small multiplicative constant.
Given quaternionic coordinates, we provide a deterministic algorithm that satisfies this upper bound in polynomial time.
For unrestricted multipliers, \(\mathrm{KLPT}^{2}\) provides a heuristic algorithm with a multiplier upper bound of order \(p^{6+o(1)}\) for arbitrary pairs and \(p^{3+o(1)}\) when one polarization matrix is the identity matrix.
Our algorithm's improvement is an explicit upper bound below \(p\) with an unconditional deterministic guarantee.
\end{abstract}

\maketitle

\section{Introduction and main theorems}

\subsection*{Isogeny-based cryptography}

Isogeny-based cryptosystems, such as CSIDH for key exchange and SQISign for signatures, rely on the hardness of the following problem:
finding isogenies between elliptic curves over finite fields of characteristic \(p\), where \(p\) is very large.
The degree of an isogeny measures its algebraic size.
Bounds on the degree thus matter for algorithms for efficiency and security.
Recently, modern isogeny-based cryptography research has expanded to higher-dimensional abelian varieties.
For supersingular elliptic curves, an application of Minkowski's theorem yields an upper bound \((2^{3/2}/\pi)p^{1/2}\).
This upper bound was obtained by Goren--Lauter \cite[Theorem 3.3]{GorenLauter2009} (a special case of their main result, see also \cite[Lemma 12]{DartoisLerouxRobertWesolowski2024}).
A refinement of the constant can be obtained by using Hermite's inequality instead, giving the upper bound \((p/2)^{1/2}\) (see \cite[Proposition 6]{KirimliKorpal2026}).
Quaternionic coordinates are extremely useful for the computation of isogenies, starting with the central example of the KLPT algorithm (\cite{KohelLauterPetitTignol2014}).
In dimension two, the relevant objects are principally polarized superspecial abelian surfaces, with polarized isogenies, and their degree is replaced by their \emph{multiplier} (the degree is the square of the multiplier).

A suitable extension of the KLPT algorithm to surfaces was developed, the KLPT\({}^{2}\) algorithm (\cite{CastryckDecruKutasLavalPetitTi2025}).
For unrestricted multipliers (i.e., not necessarily powers of a small unramified prime), KLPT\({}^{2}\) provides, in its special-extremal-order setting, a heuristic polynomial-time algorithm with multiplier \(O(p^{6+o(1)})\) for arbitrary pairs, with an improved upper bound \(O(p^{3+o(1)})\) when one polarization matrix is the identity (in general, one needs to compose two such isogenies for arbitrary pairs).
This result can be found in the extended ePrint, see \cite[Theorem 3.17]{CastryckDecruKutasLavalPetitTi2025a}.

This paper gives a direct improvement over this unrestricted-multiplier algorithm finding a small multiplier isogeny between two surfaces in quaternionic coordinates.
We replace these upper bounds on the multiplier \(N\) by
\[
N\leq \left\lfloor \frac{p}{\sqrt{2}}\right\rfloor <p,
\]
with an \emph{unconditional}, \emph{deterministic} algorithm, polynomial in the binary input length of the quaternionic data.
The bound is uniform over all pairs and maximal orders, and is optimal up to a small multiplicative constant.
It yields a separable polarized isogeny of degree at most \(p^{2}/2\) represented by an integral quaternionic matrix.

The significance of such bounds on the minimal degree of isogenies between abelian varieties is already apparent in dimension \(1\).
Aubry, Oyono, and Vincent proved that every supersingular elliptic curve admits an isogeny to its Frobenius conjugate of degree at most \((p/2)^{1/3}\) \cite[Theorem 1.1]{AubryOyonoVincent2026}.
Wesolowski used this existence theorem as a crucial abstract ingredient to design a Las Vegas heuristic algorithm to solve the supersingular isogeny problem with expected time and memory \(p^{1/3+o(1)}\) \cite[Theorem 1.1 and Corollary 1.2]{Wesolowski2026}.
This shows how even an abstract nonconstructive (constructive in quaternionic coordinates) uniform degree bound can have a definite impact on cryptanalysis.
In a recent paper, we proved an upper bound on the minimal multiplier of an isogeny between Frobenius-conjugate surfaces \cite[Theorem A]{Pham2026}, with an upper bound \((p^{3}/2)^{1/5}\).
The current paper supplies an upper bound for a complete range of endpoints, without the Frobenius-conjugacy requirement, and paves the way for several future research directions.

\subsection*{Main results}

Let \(p>3\) be a prime, let \(q=p^{2}\), and let \( k / k _{0}=\bbF_{q}/\bbF_{p}\).
Fix an algebraic closure \(\bar{ k }\).
An abelian variety over \(k\) is \emph{superspecial} if its base change to \(\bar{ k }\), as an unpolarized abelian variety, is isomorphic to a product of supersingular elliptic curves.
A \emph{principal polarization} on an abelian variety \(A\) is an isomorphism \(\lambda_{A}:A\stosim A^{\vee}\) to the dual abelian variety \(A^{\vee}\), whose base change to \(\bar{k}\) is induced by an ample line bundle.
For a homomorphism \(f:(A,\lambda_{A}) \to (C,\lambda_{C})\) between principally polarized abelian varieties, the \emph{polarized adjoint} is
\[
f^{\dagger}:C \to A,\quad f^{\dagger} = \lambda_{A}^{-1} \circ f^{\vee} \circ \lambda_{C}.
\]
A polarized isogeny has integer polarization \emph{multiplier} \(N\in \bbZ_{>0}\) if any of the following equivalent conditions holds:
\begin{equation}
\label{eq:multiplier-geometric}
f^{\vee}\lambda_{C} f = N\lambda_{A},\quad
f^{\dagger} f = [N]_{A},\quad
ff^{\dagger}=[N]_{C}.
\end{equation}
For surfaces, \(\deg(f)=N^{2}\).

\begin{theorem}
\label{theorem:A}
Let \((A,\lambda_{A})\) and \((C,\lambda_{C})\) be principally polarized superspecial abelian surfaces over \( k \), with \(p^{2}\)-Frobenius \([-p]_{A}\) and \([-p]_{C}\) respectively.
Then, there exists a separable polarized \( k \)-isogeny \(f:A \to C\) whose multiplier \(N\in\bbZ_{>0}\) satisfies
\begin{equation}
\label{eq:main-theorem-inequality}
N\leq \left\lfloor \frac{p}{\sqrt{2}}\right\rfloor,\quad
\deg(f)=N^{2}\leq \frac{p^{2}}{2}.
\end{equation}
The same multiplier and degree bounds hold over \(\bar{k}\) without a condition on models over a finite field.
\end{theorem}

Up to a small multiplicative constant, this upper bound uniform over all pairs cannot be improved.
For every prime \(p>12\) with \(p\equiv 3\pmod{4}\), we construct in Proposition \ref{proposition:lowerbound} a pair of surfaces for which every polarized isogeny has multiplier at least \(p/12\).

The proof of Theorem \ref{theorem:A} is constructive in quaternionic coordinates.
Let \(B=B_{p,\infty}\) be the quaternion algebra over \(\bbQ\) ramified at \(p\) and \(\infty\), and for \(Z\in \rmop{M}_{2}(B)\), denote by \(Z^{*}=\bar{Z}^{T}\) the quaternionic conjugate transpose.
For a hermitian matrix (i.e., \(Z^{*}=Z\)), we define the \emph{hermitian determinant} and the \emph{hermitian adjugate}
\[
\Hdet \m{a & b \\ \bar{b} & c} = ac-\nrd(b),\quad
\m{a & b \\ \bar{b} & c}^{\#}=\m{c & -b \\ -\bar{b} & a}.
\]
Then, \(ZZ^{\#}=Z^{\#}Z=\Hdet(Z)I_{2}\) and \(Z\mapsto Z^{\#}\) is linear.

The correspondence between principal polarizations on a superspecial surface and positive definite unimodular quaternion hermitian forms (the so-called \emph{IKO correspondence}, \cite{IbukiyamaKatsuraOort1986}) gives the following algorithmic form of the result.

\begin{theorem}
\label{theorem:B}
Let \(\calO\)  be a maximal order in \(B\) given by a \(\bbZ\)-basis and multiplication and conjugation tables.
Given positive definite hermitian matrices \(h_{A},h_{C}\in \rmop{M}_{2}(\calO)\) of hermitian determinant \(1\), the algorithm described in Section \ref{section:algorithm} deterministically computes a matrix \(\gamma\in \rmop{M}_{2}(\calO)\cap \rmop{GL}_{2}(B)\), and a positive integer \(N\) such that
\[
\gamma^{*}h_{C}\gamma = N h_{A},\quad
N\leq \left\lfloor \frac{p}{\sqrt{2}}\right\rfloor.
\]
Its running time is polynomial in the total binary length of the input data.
The matrix \(\gamma\) represents a separable isogeny \(A\to C\) of degree \(N^{2}\).
\end{theorem}

Throughout, \(\trd\) and \(\nrd\) denote reduced trace and norm on \(B\), and \(\Trd\) denotes reduced trace on \(\rmop{M}_{2}(B)\).
For a quadratic lattice \((L,q)\), we adopt the following determinant convention:
\begin{equation}
\label{eq:determinant-convention}
\det(L,q)= \det\bigl(B_{q}(e_{i},e_{j})\bigr)_{1\leq i,j\leq n},\quad
B_{q}(x,y)=\frac{q(x+y)-q(x)-q(y)}{2},
\end{equation}
where  \((e_{i})\) is a \(\bbZ\)-basis of \(L\).

\section{Quaternionic coordinates and the rank-five lattice}

\subsection*{Fields of definition}

We first record the following standard facts on rationality properties of homomorphisms.

\begin{lemma}
\label{lemma:basic-rationality-integrality}
Let \(Y\) and \(Z\) be abelian varieties over \( k \) with \(p^{2}\)-Frobenius equal to \([-p]_{Y}\) and \([-p]_{Z}\) respectively.
Then, every geometric homomorphism \(Y_{\bar{k}}\to Z_{\bar{k}}\) is defined over \(k\).
Moreover, for every positive-dimensional abelian variety \(Y\) over a field,
\begin{equation}
\label{eq:integral-scalar-endomorphisms}
\End_{k}(Y)\cap (\bbQ\rmop{id}_{Y}) = \bbZ \rmop{id}_{Y}\quad \text{inside}~ \End_{k}(Y)\otimes_{\bbZ}\bbQ.
\end{equation}
\end{lemma}

\begin{proof}
The first claim is classical and follows from, e.g., \cite[\S 5]{JordanKeetonPoonenRainsShepherdBarronTate2018}.
For the second, let \(u=(a/b)\rmop{id}_{Y}\in \rmop{End}_{k}(Y)\) with \(a,b\in \bbZ\), \(b>0\), and \(\rmop{gcd}(a,b)=1\).
Since \(\rmop{End}_{k}(Y)\) is a torsion-free abelian group (see \cite[Lemma 12.2]{Milne1986}), we can clear denominators and \(u\circ [b]_{Y} = [b]_{Y}\circ u = [a]_{Y}\) as endomorphisms.
If \(R\) is a commutative \(k\)-algebra, we can restrict this to \(Y[b](R)\) where \(Y[b]=\rmop{ker}[b]_{Y}\).
This shows that for any \(k\)-algebra \(R\) and \(t\in Y[b](R)\), \([a]_{Y}(t)=u(0)=0\), so \(t\) is annihilated by \(a\) and \(b\).
Let \(r\) and \(s\) be integers such that \(ra+sb=1\).
Then, \(t=[ra+sb]_{Y}(t)=[r]_{Y}([a]_{Y}(t))+[s]_{Y}([b]_{Y}(t))=0\), so \(Y[b]\) is the trivial group scheme.
But \([b]_{Y}\) is an isogeny of degree \(b^{2\dim Y}\) (\cite[Theorem 8.2]{Milne1986}).
Since \(Y[b]\) is trivial, this forces \(b=1\).
\end{proof}

\subsection*{Coordinates}

Let \(E\) be a supersingular elliptic curve over \( k \) with \(p^{2}\)-Frobenius equal to \([-p]_{E}\).
Such a curve exists (see \cite[\S 5]{JordanKeetonPoonenRainsShepherdBarronTate2018}).
Put
\[
\calO=\End_{k}(E)=\End_{\bar{k}}(E),\quad
B=\calO\otimes_{\bbZ}\bbQ.
\]
By Deuring's theorem, \(B=B_{p,\infty}\) is the quaternion algebra over \(\bbQ\) ramified at \(p\) and infinity, and \(\calO\) is a maximal order (see \cite[Theorem 42.1.9]{Voight2021}).
We write \(a\mapsto \bar{a}\) for the canonical quaternion involution, so that \(\trd(a)=a+\bar{a}\) and \(\nrd(a)=a\bar{a}\) for \(a\in B\).
They are integer-valued on \(\calO\) and \(\calO\cap \bbQ=\bbZ\).

Every superspecial abelian surface has a model over \(k\) with \(p^{2}\)-Frobenius \([-p]\), and as an unpolarized surface, it is isomorphic to \(E_{\bar{k}}^{2}\).
For principally polarized superspecial abelian varieties \((Y,\lambda_{Y})\), \(Y=A,C\), fix unpolarized \(k\)-isomorphisms \(\iota_{Y}:Y \to E^{2}\), and let \(\lambda_{0}\) denote the product principal polarization on \(E^{2}\).
The isomorphism \(\iota_{Y}:Y \to E^{2}\) induces an isomorphism of abelian groups
\[
\Hom_{k}(Y,Y^{\vee}) \stosim \Hom_{k}(E^{2},(E^{2})^{\vee}),\quad
\alpha \mapsto (\iota_{Y}^{-1})^{\vee} \circ \alpha \circ \iota_{Y}^{-1}.
\]
Define \(h_{Y}:E^{2} \to E^{2}\) by
\[
h_{Y}=\lambda_{0}^{-1} \circ (\iota_{Y}^{-1})^{\vee} \circ \lambda_{Y} \circ \iota_{Y}^{-1}\in \End_{k}(E^{2})=\rmop{M}_{2}(\calO).
\]
The polarized surface in these coordinates is therefore \((E^{2},\lambda_{0}h_{Y})\).
A homomorphism \(A\to C\) in these coordinates is thus \((E^{2},\lambda_{0} h_{A})\to (E^{2},\lambda_{0}h_{C})\), and the IKO correspondence identifies \(\lambda_{Y}\) with \(h_{Y}\).

Let \(f:A \to C\)  be a homomorphism.
In coordinates, \(f\) corresponds to a matrix \(\gamma:E^{2} \to E^{2}\) making the square
\[
\begin{tikzcd}
A \dar[swap]{\iota_{A}}\rar{f} & C \dar{\iota_{C}}\\
E^{2} \rar[swap]{\gamma} & E^{2}
\end{tikzcd}
\]
commute.
This change of coordinates is an isomorphism of abelian groups
\begin{equation}
\label{eq:morphism-coordinate-matrix}
\Hom_{k}(A,C) \stosim \End_{k}(E^{2})\cong \rmop{M}_{2}(\calO),\quad
f\mapsto \gamma = \iota_{C}\circ f\circ \iota_{A}^{-1}.
\end{equation}
In particular, we have a ring isomorphism
\[
\End_{k}(Y) \stosim \rmop{M}_{2}(\calO),\quad
u\mapsto \iota_{Y} \circ u \circ \iota_{Y}^{-1}\qquad (Y=A,C).
\]
Tensoring with \(\bbQ\) we obtain rational homomorphisms and the algebra \(\rmop{M}_{2}(B)\).

\subsection*{Hermitian matrices}

For \(Z=(z_{ij})\in \rmop{M}_{2}(B)\), we write \(Z^{*}\) for the quaternionic conjugate transpose.
We denote by \(\Herm_{2}(B)\) (resp. \(\Herm_{2}(\calO)\)) the set of hermitian matrices in \(\rmop{M}_{2}(B)\) (resp. \(\rmop{M}_{2}(\calO)\)).
This anti-involution corresponds to the polarized adjoint for the product polarization \(\lambda_{0}\).
For a matrix \(Z\in \rmop{M}_{2}(\calO)\), this implies
\begin{equation}
\label{eq:dual-conjugate}
\lambda_{0}^{-1}\circ Z^{\vee} \circ \lambda_{0} = Z^{*},\quad
Z^{\vee} \circ \lambda_{0} = \lambda_{0} \circ Z^{*}.
\end{equation}

We state the IKO correspondence, see \cite[Proposition 2.8 and Corollary 2.9]{IbukiyamaKatsuraOort1986}.

\begin{proposition}
For \(Y=A,C\), the matrix \(h_{Y}\) is positive definite and hermitian, and its hermitian determinant is one.
More precisely,
\[
h_{Y} = \begin{pmatrix}
s & r \\ \bar{r} & t
\end{pmatrix}
\in \rmop{M}_{2}(\calO),\quad s,t\in \bbZ_{>0},\quad r\in \calO,\quad st-\nrd(r)=1.
\]
Its inverse is
\[
h_{Y}^{-1} = \m{ t & -r \\ -\bar{r} & s}\in \rmop{M}_{2}(\calO).
\]
Conversely, every such matrix in \(\rmop{M}_{2}(\calO)\) defines a principal polarization \(\lambda_{0}h_{Y}\) on \(E^{2}\).
\end{proposition}

Let \(f:A \to C\) be a homomorphism with corresponding matrix \(\gamma\in \rmop{M}_{2}(\calO)\).
Its polarized adjoint in these coordinates corresponds to \(h_{A}^{-1}\gamma^{*}h_{C}\).
Therefore, the multiplier equation for the isogeny \(f:A \to C\),
\[
\begin{tikzcd}
A\dar[swap]{N\lambda_{A}} \rar{f} & C \dar{\lambda_{C}}\\
A^{\vee} & C^{\vee} \lar{f^{\vee}}
\end{tikzcd}
\quad
f^{\vee}\lambda_{C} f = N\lambda_{A},
\]
becomes
\begin{equation}
\label{eq:multiplier-matrix}
f^{\vee} \lambda_{C} f = N\lambda_{A} 
\quad\Leftrightarrow\quad
\gamma^{*}h_{C}\gamma=Nh_{A}
\quad\Leftrightarrow\quad
\gamma^{\dagger}\gamma=NI_{2}.
\end{equation}
Conversely, an integral matrix \(\gamma\in \rmop{M}_{2}(\calO)\) satisfying \eqref{eq:multiplier-matrix} defines \(f=\iota_{C}^{-1}\gamma \iota_{A}\) such that \(f^{\dagger}f=[N]_{A}\).
The kernel \(\ker f\) is a group scheme contained in \(A[N]\), hence it is finite.
Since the two surfaces have the same dimension, \(f\) is an isogeny of degree \(N^{2}\).
We have reduced the problem of determining the polarized isogeny to a matrix equation.

\subsection*{An initial isogeny}

\begin{proposition}
\label{proposition:existence-multipliers}
There exists a polarized isogeny \(C\to A\) with positive integer multiplier.
\end{proposition}

\begin{proof}
Let \(S=h_{A}+h_{C}\), and let \(D=\Hdet(S)\).
Define \(\Psi=h_{A}^{-1}S=I_{2}+h_{A}^{-1}h_{C}\).
Then, \(\Psi\in \rmop{M}_{2}(\calO)\) since \(h_{A}^{-1},h_{C}\in \rmop{M}_{2}(\calO)\).
The matrix \(S\) is positive definite since \(h_{A}\) and \(h_{C}\) are, so \(D\) is a positive integer, and \(\Psi\in \rmop{GL}_{2}(B)\), so \(\Psi\) is invertible over \(B\).
Recall that the hermitian adjugate of a hermitian matrix \(h\in \rmop{M}_{2}(B)\) is
\[
h^{\#}=\m{c & - b \\ - \bar{b} & a},\quad\text{for}~ h=\m{a & b \\ \bar{b} & c},
\]
and \(hh^{\#}=h^{\#}h=\Hdet(h)I_{2}\).

By linearity of \(S\mapsto S^{\#}\), we obtain \(S^{\#}=h_{A}^{\#}+h_{C}^{\#}=h_{A}^{-1}+h_{C}^{-1}\).
Thus, we have \(S=h_{A}\Psi\) and \(S^{\#} h_{C}=\Psi\).
Since \(S^{\#}\) and \(h_{C}\) are hermitian, this gives \(\Psi^{*}=h_{C}S^{\#}\).
Therefore,
\[
\Psi^{*}h_{A}\Psi=h_{C}S^{\#}S=D\,h_{C}.\qedhere
\]
\end{proof}

\begin{lemma}
Let \(h\in \rmop{M}_{2}(\calO)\) and \(H\in \rmop{M}_{2}(B)\) be hermitian with \(\Hdet(h)=1\).
Write \(h=\tm{s & r \\ \bar{r} & t}\), \(H=\tm{a & b \\ \bar{b} & c}\).
Define
\begin{equation}
\label{eq:eta-H}
\eta_{h}(H) =  ta+sc-\trd(\bar{r}b),\quad X=h^{-1}H.
\end{equation}
Then,
\begin{align}
\label{eq:cross-adjugate}
Hh^{\#}+hH^{\#} & = \eta_{h}(H) I_{2},\\
\label{eq:Trd-eta}
\Trd(X) & = 2\eta_{h}(H),\\
\label{eq:characteristic-pol}
X^{2}-\eta_{h}(H)X+\Hdet(H)I_{2}&=0.
\end{align}
\end{lemma}

\begin{proof}
\eqref{eq:cross-adjugate} and \eqref{eq:Trd-eta} are simple computations.
For \eqref{eq:characteristic-pol}, multiply \eqref{eq:cross-adjugate} on the left by \(h^{-1}=h^{\#}\) and on the right by \(H\).
The left-hand side is
\[
h^{-1}Hh^{\#}H+H^{\#}H=(h^{-1}H)(h^{-1}H)+\Hdet(H)I_{2}=X^{2}+\Hdet(H)I_{2}.
\]
The right-hand side becomes \(\eta_{h}(H)h^{-1}H=\eta_{h}(H)X\).
This proves \eqref{eq:characteristic-pol}.
\end{proof}

Fix the surface \(A\) and let \(h=h_{A}\).
Define
\begin{align*}
S_{h} & = \{X \in \rmop{M}_{2}(B) \mid h^{-1}X^{*}h=X,\, \Trd(X)=0\},\\
\scrS_{h} & = S_{h} \cap \rmop{M}_{2}(\calO),\\
q_{0}(X)
& = \frac{1}{4} \Trd(X^{2}).
\end{align*}
We may write \(\scrS_{A}\) for the corresponding lattice in \(\End_{k}(A)\).
It is well-defined since it does not depend on the choice of coordinates.
It can also be defined directly on the principally polarized superspecial abelian surface \((A,\lambda_{A})\).
Write \(\End_{k}(A)\) for the integral endomorphisms and \(D_{A}=\End_{k}(A)_{\bbQ}=\End_{k}(A)\otimes_{\bbZ}\bbQ\).
The algebra \(D_{A}\) is a central simple algebra over \(\bbQ\).
Then, the lattice \(\scrS_{A}\) is defined by
\begin{equation}
\label{eq:SA-lattice}
\scrS_{A} = \{x\in \End_{k}(A) \mid x^{\dagger}=x,\, \Trd_{D_{A}/\bbQ}(x)=0\}.
\end{equation}
The quadratic form intrinsically defined from \(q_{0}\) is now
\begin{equation}
\label{eq:SA-lattice-QF}
q_{A}(x) = \frac{1}{4} \Trd_{D_{A}/\bbQ}(x^{2}).
\end{equation}
The quadratic lattice \((\scrS_{A},q_{A})\) is positive definite, integral, of rank five.
This will follow from the same properties of \(\scrS_{h}\).
To be explicit, we have a coordinate isomorphism \(\iota_{A}:A \stosim E^{2}\), and it easily follows that \(q_{A}(x)=q_{0}(\iota_{A}x\iota_{A}^{-1})\) for any \(x\in \scrS_{A}\otimes_{\bbZ}\bbQ\).
Then conjugation by \(\iota_{A}\) gives an isometry \((\scrS_{A},q_{A})\stosim (\scrS_{h},q_{0})\).

\begin{proposition}
\label{prop:basic-quadratic-form}
The \(\bbQ\)-vector space \(S_{h}\) is five-dimensional, and \(\scrS_{h}\) is a full lattice.
The form \(q_{0}\) is positive definite and integer-valued on \(\scrS_{h}\).
For every \(X\in S_{h}\), we have
\begin{equation}
\label{eq:self-adjoint-quadratic-form}
X^{2} = q_{0}(X)\, I_{2},\quad \text{where}~ 
q_{0}(X) = - \Hdet(hX).
\end{equation}
Let \(\Lambda\), \(Q:\Lambda \to \bbZ\), and \(\eta_{h}:\Lambda \to \bbZ\) be defined by
\begin{align}
\label{eq:Lambda-Q-def}
\Lambda
& = \bbZ \oplus \calO\oplus \bbZ,\\
\label{eq:Q-def}
Q(a,b,c)
& = \nrd(b)-ac,\quad (a,b,c)\in \Lambda,\\
\label{eq:eta-def}
\eta_{h}(a,b,c)
& = ta+sc-\trd(\bar{r}b),\quad (a,b,c)\in \Lambda.
\end{align}
Let \(K_{h}=\ker \eta_{h}\subset \Lambda\).
Then, multiplication by \(h\) induces an isometry of integral quadratic lattices
\begin{equation}
\label{eq:Phi-h-isometry}
\Phi_{h}:(\scrS_{h},q_{0}) \stosim (K_{h},Q|_{K_{h}}),\quad
X\mapsto\bigl((hX)_{11},(hX)_{12},(hX)_{22}\bigr).
\end{equation}
\end{proposition}

\begin{proof}
The identity \(X^{*}h=hX\) is equivalent to \(H=hX\) being hermitian.
Recall that we had defined \(\eta_{h}\) on the space of hermitian matrices in \eqref{eq:eta-H}.
By \eqref{eq:Trd-eta}, \(\Trd(X)=0\) if and only if \(\eta_{h}(H)=0\).
The space of hermitian matrices has dimension \(6\) over \(\bbQ\).
Vanishing of \(\eta_{h}(H)\) then implies \(\dim_{\bbQ}S_{h}=5\).
Since \(\rmop{M}_{2}(\calO)\) is a full lattice in \(\rmop{M}_{2}(B)\), intersecting it with a rational subspace gives a full lattice, so \(\scrS_{h}\) is a full lattice in \(S_{h}\).
By \eqref{eq:characteristic-pol}, on \(S_{h}\), we have \(X^{2}=-\Hdet(H)I_{2}\).
Taking reduced traces, \(\Trd(X^{2})=4(-\Hdet(H))\), which proves \eqref{eq:self-adjoint-quadratic-form}.

Let \(X\in S_{h}\) be nonzero, and let \(v\in B^{2}\) with \(Xv\neq 0\).
Then, since \(h\) is positive definite, we have
\[
0<(Xv)^{*}h(Xv)=v^{*}(X^{*}hX)v=v^{*}(hX^{2})v=q_{0}(X)(v^{*}hv),
\]
where we have used that \(X^{2}\) is a scalar matrix.
Since \(h\) is positive definite, we also have \(v^{*}hv>0\), hence \(q_{0}(X)>0\), so \(q_{0}\) is positive definite.

Since \(h,h^{-1}\in \rmop{M}_{2}(\calO)\), multiplication by \(h\) is an automorphism of \(\rmop{M}_{2}(\calO)\).
Multiplication by \(h\) induces an isomorphism
\[
\{X\in \rmop{M}_{2}(\calO)\mid X^{*}h=hX,\,\Trd(X)=0\}
\cong\rmop{Herm}_{2}(\calO)\cap \rmop{ker}\eta_{h}.
\]
Thus, we have an isomorphism of \(\bbZ\)-modules
\[
\Phi_{h}:\scrS_{h} \stosim K_{h},\quad
X\mapsto  hX=\m{a & b \\ \bar{b} & c} \mapsto (a,b,c),
\]
which is the map \eqref{eq:Phi-h-isometry}, with inverse \((a,b,c)\mapsto h^{-1}\tm{a & b \\ \bar{b} & c}\).
This is an isometry.
Indeed, since \(hX\) is hermitian, we have
\begin{align*}
Q(\Phi_{h}(X))
& = Q\bigl((hX)_{11},(hX)_{12},(hX)_{22}\bigr)\\
& = \nrd((hX)_{12})-(hX)_{11}(hX)_{22}
 = -\Hdet(hX)
 = q_{0}(X).
\end{align*}
Since \(Q\) is clearly integral on \(\Lambda\), this shows that \(q_{0}\) is integral on \(\scrS_{h}\).
\end{proof}

Transporting this by \eqref{eq:SA-lattice} and \eqref{eq:SA-lattice-QF} shows that \((\scrS_{A},q_{A})\) is a positive integral quadratic lattice of rank five, with
\[
x^{2}=q_{A}(x) \rmop{id}_{A}\quad (x\in \scrS_{A}\otimes_{\bbZ}\bbQ).
\]

We record basic facts on \(\calO\).

\begin{lemma}
\label{lemma:determinant-O-nrd}
The quaternion norm lattice satisfies
\begin{equation}
\label{eq:determinant-O-nrd}
\det(\calO,\nrd)=\frac{p^{2}}{16}.
\end{equation}
Let \(\frakP\) be the unique two-sided prime ideal of \(\calO\) above \(p\), and let \(\calO^{\#}=\{z\in B \mid \trd(z\calO)\subset \bbZ\}\) be the codifferent.
Then,
\[
\calO^{\#}=p^{-1}\frakP,\quad
\frakP^{2}=p\calO,\quad
\bar{\frakP}=\frakP,\quad
\nrd(\frakP)\subset p\bbZ.
\]
\end{lemma}

\begin{proof}
For \(u,v\in B\), we have \(b_{\nrd}(u,v)=\frac{1}{2}\trd(u\bar{v})\).
Let \((e_{i})_{1\leq i\leq 4}\) be a \(\bbZ\)-basis of \(\calO\) and let \(T=(\trd(e_{i}e_{j}))_{i,j}\).
Since \(\calO\) is a maximal order, its reduced discriminant is \(p\) (see \cite[Theorem 15.5.5]{Voight2021}), the squarefree discriminant of \(B\), and the nonreduced discriminant of \(\calO\), which is equal to \(|\det T|\), is \(p^{2}\), hence \(|\det T|=p^{2}\).
Since quaternion conjugation is a linear automorphism, and an involution, let \(J\in \rmop{GL}_{4}(\bbZ)\) denote its matrix, where \(J^{2}=I_{4}\).
Hence, \(\det J=\pm 1\) and \(\bar{e}_{j}=\sum_{k} J_{kj}e_{k}\).
Therefore,
\[
|\det(b_{\nrd}(e_{i},e_{j}))|=\bigl|\det\bigl(\tfrac{1}{2}TJ\bigr)\bigr|
=2^{-4} \, |\det T|\cdot |\det J| = \frac{p^{2}}{16}.
\]
Since \(\nrd\) is positive definite, this determinant is \(p^{2}/16\).

The remaining assertions are standard properties of the different and two-sided ideals of maximal quaternion orders, see \cite[\S 16.8]{Voight2021}.
\end{proof}

\begin{proposition}
The quadratic lattices \((\scrS_{A},q_{A})\) and \((\scrS_{h},q_{0})\) have determinant
\[
\det(\scrS_{A},q_{A})=\det(\scrS_{h},q_{0})=\frac{p^{2}}{16}.
\]
\end{proposition}

\begin{proof}
The quadratic lattice \((\Lambda,Q)\) decomposes as an orthogonal sum
\[
(\Lambda,Q)\simeq (\calO,\nrd)\perp(\bbZ^{2},(a,c)\mapsto -ac).
\]
By \eqref{eq:determinant-O-nrd},
\[
\det(\Lambda,Q)=\det(\calO,\nrd)\cdot \det(\bbZ^{2},(a,c)\mapsto -ac)=\left(\frac{p^{2}}{16}\right)\cdot\left(-\frac{1}{4}\right)=-\frac{p^{2}}{64}.
\]

We now prove that \(\eta_{h}:\Lambda \to \bbZ\) is surjective.
It is nonzero, since \(\eta_{h}(1,0,0)=t>0\).
Write \(\eta_{h}(\Lambda)=d\bbZ\) with \(d>0\).
Then, evaluating \(\eta_{h}\) on the three coordinates independently, we get \(s,t\in d\bbZ\) and \(\trd(\bar{r}\calO)\subset d\bbZ\).
So there are \(s_{0},t_{0}\in \bbZ\) such that \(s=ds_{0}\) and \(t=dt_{0}\) and \(\bar{r}/d\in \calO^{\#}\).
By Lemma \ref{lemma:determinant-O-nrd} and invariance of \(\frakP\) under conjugation, there exists \(v\in \frakP\) such that \(r/d=v/p\).
Since \(\nrd(\frakP)\subset p\bbZ\), write \(\nrd(v)=pn_{0}\) with \(n_{0}\in \bbZ\).
Then, \(r=(d/p)v\), so \(\Hdet(h)=1\) implies \(1=st-\nrd(r)=d^{2}s_{0}t_{0} - (d^{2}/p)n_{0}\), hence \(p=d^{2}(ps_{0}t_{0}-n_{0})\).
Thus, \(d^{2}\mid p\).
Since \(p\) is prime, \(d=1\), so \(\eta_{h}\) is surjective.

Let \(w=(s,r,t)\in \Lambda\).
Then, \(Q(w)=-1\).
A direct calculation shows that \(2B_{Q}(w,v) = - \eta_{h}(v)\).
Thus, \(K_{h}=\Lambda \cap w^{\perp}\), and \(\eta_{h}(w)=2\).
Under the isomorphism \(\Lambda/K_{h}\stosim \bbZ\), \(v+K_{h}\mapsto \eta_{h}(v)\), the image of \(\bbZ w\) is \(2\bbZ\).
Since also \(Q(w)\neq 0\), the sum \(K_{h}+\bbZ w\) is direct, and we conclude that \(K_{h}\oplus \bbZ w\) is an index \(2\) sublattice.
Thus,
\[
\det(K_{h}\oplus \bbZ w,Q)
=\det(K_{h},Q|_{K_{h}})\cdot Q(w)
=-\det(K_{h},Q|_{K_{h}}),
\]
while \(\det(K_{h}\oplus \bbZ w,Q)=2^{2}\det(\Lambda,Q) = - \frac{p^{2}}{16}\).
We conclude that \(\det(K_{h},Q)=p^{2}/16\), and using the isometry \(\Phi_{h}\) from \eqref{eq:Phi-h-isometry} completes the proof.
\end{proof}

\section{A lattice for the prescribed pair}

Let \(\psi:C \to A\) be a polarized isogeny with positive integer multiplier \(M\).
Thus, by \eqref{eq:multiplier-geometric}, \(\psi^{\dagger}\psi = [M]_{C}\) and \(\psi \psi^{\dagger}= [M]_{A}\).
Define
\begin{equation}
\label{eq:quadratic-lattices-definition}
I_{\psi} = \psi \Hom_{k}(A,C),\quad
L_{\psi}=\scrS_{A} \cap I_{\psi},\quad
q_{\psi} = \frac{1}{M}q_{A}|_{L_{\psi}}.
\end{equation}
In coordinates, the isogeny \(\psi\) corresponds to the matrix \(\Psi=\iota_{A}\circ \psi\circ \iota_{C}^{-1} \in \rmop{M}_{2}(\calO)\) which satisfies
\begin{equation}
\label{eq:multiplier-matrix-2}
\Psi^{*}h_{A}\Psi = M h_{C}.
\end{equation}
For \(f:A \to C\), now put \(\gamma = \iota_{C}\circ f \circ \iota_{A}^{-1}\).
The coordinate matrix of \(x=\psi\circ f\) is then given by
\begin{equation}
\label{eq:coordinate-composition}
\begin{tikzcd}
A\dar[swap]{\iota_{A}} \rar{\psi \circ f} & A\dar{\iota_{A}}\\
E^{2} \rar[swap]{\Psi\gamma} & E^{2}
\end{tikzcd}
\qquad
X= \iota_{A} \circ x \circ\iota_{A}^{-1} = \Psi \gamma.
\end{equation}
To keep track of coordinates, let \(\kappa_{A,C}\), \(\kappa_{A}\) be defined by
\begin{equation}
\label{eq:kappa-coordinates}
\kappa_{A,C}(f) = \iota_{C} \circ f \circ \iota_{A}^{-1},\quad
\kappa_{A}(x) = \iota_{A} \circ x \circ \iota_{A}^{-1}.
\end{equation}
We then have a commutative diagram of abelian groups:
\[
\begin{tikzcd}[column sep=5em]
\Hom_{k}(A,C) \dar{\cong}[swap]{\kappa_{A,C}}\rar{f\mapsto \psi \circ f} & \End_{k}(A) \dar{\kappa_{A}}[swap]{\cong}\\
\rmop{M}_{2}(\calO) \rar[swap]{\gamma \mapsto \Psi \gamma} & \rmop{M}_{2}(\calO).
\end{tikzcd}
\]
Thus,
\[
\kappa_{A}(I_{\psi})=\Psi \rmop{M}_{2}(\calO),\quad
\kappa_{A}(L_{\psi})=S_{h_{A}} \cap \Psi \rmop{M}_{2}(\calO).
\]
Regarding the multiplier equation, we now have in these coordinates
\begin{equation}
\label{eq:multiplier-matrix-composition}
f^{\dagger}f=[N]_{A}
\quad\Leftrightarrow\quad
X^{\dagger_{h_{A}}}X = MN\, I_{2}
\quad\Leftrightarrow\quad
x^{\dagger}x = [MN]_{A}.
\end{equation}
If \(X\in S_{h_{A}}\), then \(X\) is self-adjoint with respect to \(\dagger_{h_{A}}\).
The multiplier equation now becomes \(X^{\dagger_{h_{A}}}X=X^{2}=q_{0}(X)\,I_{2}\), that is,
\begin{equation}
\label{eq:multiplier-composition-qf}
x^{\dagger}x = [q_{A}(x)]_{A}.
\end{equation}
Thus, equation \eqref{eq:multiplier-matrix-composition} becomes a scalar equation
\[
q_{0}(X) = MN.
\]
Thus, we need the multiplier of \(f\) to be \(q_{0}(X)/M\).
We prove in the next proposition that \(q_{0}(X)/M\) is always an integer for \(X\in \kappa_{A}(L_{\psi})\) nonzero.

\begin{proposition}
\label{proposition:factorization-sublattice}
The pair \((L_{\psi},q_{\psi})\) is a positive integral quadratic lattice of rank five.
Every nonzero \(x\in L_{\psi}\) has a unique expression \(x= \psi \circ f\) with \(f\in \Hom_{k}(A,C)\), and this homomorphism is an isogeny satisfying
\begin{equation}
\label{eq:final-multiplier}
f^{\dagger} f = [q_{\psi}(x)]_{A}.
\end{equation}
Furthermore,
\begin{equation}
\label{eq:finite-index-sublattice-composition}
M\, \scrS_{A} \subset L_{\psi} \subset \scrS_{A},\quad
\det(L_{\psi},q_{\psi}) = \frac{p^{2}}{16M^{5}} [\scrS_{A}:L_{\psi}]^{2}.
\end{equation}
\end{proposition}

\begin{proof}
Let \(x\in \scrS_{A}\).
From \(\psi:C\to A\), we form the homomorphism \(\psi^{\dagger}\circ x \in \Hom_{k}(A,C)\).
Then, the multiplier equation for \(\psi\) shows that we have an endomorphism \(\psi\circ(\psi^{\dagger}\circ x)=Mx\) which belongs to \(\psi \Hom_{k}(A,C)=I_{\psi}\).
Since \(x\in \scrS_{A}\), we also have \(Mx\in \scrS_{A}\), hence \(Mx\in L_{\psi}\).
This proves that \(M\scrS_{A}\subset L_{\psi}\), the first inclusion in \eqref{eq:finite-index-sublattice-composition}.
Since \(M\scrS_{A}\) has finite index in \(\scrS_{A}\), we conclude that \(L_{\psi}\) has finite index in \(\scrS_{A}\), hence also has rank five.

The multiplier equation with multiplier \(q_{\psi}(x)\) also follows from the definition of \(q_{\psi}\) and remarks preceding the proposition.
Indeed, for \(x=\psi f \in L_{\psi}\), self-adjointness, \eqref{eq:multiplier-composition-qf}, and \eqref{eq:multiplier-matrix-composition} give
\[
[q_{A}(x)]_{A}=x^{\dagger}x = f^{\dagger} \psi^{\dagger} \psi f = M f^{\dagger} f.
\]
Therefore, \(f^{\dagger}f=(q_{A}(x)/M)\rmop{id}_{A}\) inside \(\End_{k}(A)_{\bbQ}\).
But since \(f^{\dagger}f\in \End_{k}(A)\), Lemma \ref{lemma:basic-rationality-integrality} implies that \(f^{\dagger}f\in \End_{k}(A)\cap (\bbQ\rmop{id}_{A})=\bbZ\rmop{id}_{A}\) inside \(\End_{k}(A)_{\bbQ}\).
We conclude that \(q_{\psi}(x)=q_{A}(x)/M\in \bbZ\).
The rest of the proof is routine.
The form \(q_{\psi}\) is positive for \(x\neq 0\) since \(q_{0}\) is positive definite, and \(\ker f\) as a group scheme is contained in \(A[q_{\psi}(x)]\) by \eqref{eq:final-multiplier}, hence this kernel is a finite group scheme.
Since the abelian varieties have the same dimension, this is an isogeny.
The uniqueness of \(f\) can be argued as follows.
Since \(\psi\) has multiplier \(M\), we can define \(\psi^{-1}=\frac{1}{M}\psi^{\dagger}\) as a two-sided inverse with respect to composition.
The homomorphism \(\psi^{-1}\) belongs to \(\Hom_{k}(A,C)\otimes_{\bbZ}\bbQ\).
Then, if \(f:A\to C\) satisfies \(x=\psi\circ f\), then it must be that \(f=\psi^{-1}\circ x=\frac{1}{M} \psi^{\dagger}\circ x\).
Thus, \(f\) is unique in \(\Hom_{k}(A,C)\) which, since it is torsion-free, embeds in \(\Hom_{k}(A,C)\otimes_{\bbZ}\bbQ\).

Finally, \(L_{\psi}\) is a sublattice of index \([\scrS_{A}:L_{\psi}]\), so the Gram determinant is multiplied by \([\scrS_{A}:L_{\psi}]^{2}\), whereas \(q_{\psi}\) is \(q_{A}\) multiplied by \(M^{-1}\), hence the Gram determinant is multiplied by \(M^{-5}\) in rank five.
We conclude that
\[
\det(L_{\psi},q_{\psi})= \frac{[\scrS_{A}:L_{\psi}]^{2}}{M^{5}}\, \det(\scrS_{A},q_{A})
=\frac{[\scrS_{A}:L_{\psi}]^{2}}{M^{5}}\, \frac{p^{2}}{16}.\qedhere
\]
\end{proof}

The immediate upper bound \([\scrS_{A}:L_{\psi}]\leq M^{5}\) is not enough for our purposes, so our goal is to find an improved upper bound on the index \([\scrS_{A}:L_{\psi}]\).

\begin{proposition}
\label{proposition:index-bound}
Let \(\psi:C \to A\) be a polarized isogeny with positive integer multiplier \(M\).
Then,
\begin{equation}
\label{eq:lattice-index-global-upper-bound}
[\scrS_{A}:L_{\psi}]\leq M^{4},\quad
\det(L_{\psi},q_{\psi}) \leq \frac{p^{2}M^{3}}{16}.
\end{equation}
\end{proposition}

The proof will occupy the rest of this section and follow from a sequence of lemmas and propositions.
We need some additional notation.
For a prime \(\ell\), let \(\calO_{\ell}=\calO\otimes_{\bbZ}\bbZ_{\ell}\), and use a subscript \(\ell\) for extension of scalars from \(\bbZ\) to \(\bbZ_{\ell}\).
In the coordinates given by \(\kappa_{A,C}\) and \(\kappa_{A}\) from \eqref{eq:kappa-coordinates}, we have
\begin{align*}
\scrS_{A,\ell} = \scrS_{A}\otimes_{\bbZ}\bbZ_{\ell} & = \{X\in \rmop{M}_{2}(\calO_{\ell}) \mid X^{\dagger_{h_{A}} }= X,\, \Trd(X)=0\},\\
L_{\psi,\ell}  = L_{\psi}\otimes_{\bbZ}\bbZ_{\ell}& = \scrS_{A,\ell} \cap \Psi \rmop{M}_{2}(\calO_{\ell}),\\
[\scrS_{A}:L_{\psi}] & = \prod_{\ell} [\scrS_{A,\ell}:L_{\psi,\ell}].
\end{align*}
In the product, the local factor equals \(1\) whenever \(\ell\nmid M\).
See \cite[\S 81, in particular 81:3, 81:11--81:12, \S 81E]{OMeara1963}.
Alternatively, see \cite[Chap. II, \S 4, (4.7), (4.7.a), (4.9), (4.10), pp.~92--93]{FroehlichTaylor1993}.
For the formula of the index as the product of the local indices, see \cite[Chapter 11, Lemma 2.1]{Cassels1978}.

The next elementary lemma, whose proof we omit, lets us keep track when changing bases.

\begin{lemma}
\label{lemma:pairs-local-lattices}
Let \(\scrT_{\ell}=\rmop{M}_{2}(\calO_{\ell})\).
Assume that \(h_{A},h_{C}\in \scrT_{\ell}^{\times}\) are hermitian and that \(\Psi^{*}h_{A}\Psi=M h_{C}\).
Given \(U_{A},U_{C}\in \scrT_{\ell}^{\times}\), put \(h_{Y}^{\prime} = U_{Y}^{*} h_{Y}U_{Y} \; (Y=A,C)\), \(\Psi^{\prime}=U_{A}^{-1} \Psi U_{C}\).
Then,
\[
(\Psi^{\prime})^{*} h_{A}^{\prime} \Psi^{\prime} = M h_{C}^{\prime}, 
\]
and conjugation \(X\mapsto X^{\prime}=U_{A}^{-1}XU_{A}\) identifies the pairs of lattices
\[
(\scrS_{h_{A},\ell},\scrS_{h_{A},\ell}\cap \Psi \scrT_{\ell})\cong
(\scrS_{h_{A}^{\prime},\ell},\scrS_{h_{A}^{\prime},\ell}\cap \Psi^{\prime}\scrT_{\ell}),
\]
where for \(h\in \scrT_{\ell}^{\times}\),
\[
\scrS_{h,\ell}=\{X\in \rmop{M}_{2}(\calO_{\ell}) \mid X^{*}h=hX,\, \Trd(X)=0\}.
\]
\end{lemma}

Observe that when \(\ell\) is unramified in \(B\), \(\calO_{\ell}=\rmop{M}_{2}(\bbZ_{\ell})\), hence \(\rmop{M}_{2}(\calO_{\ell})=\rmop{M}_{4}(\bbZ_{\ell})\), whence \(\scrT_{\ell}^{\times}=\rmop{GL}_{4}(\bbZ_{\ell})\).
When \(\ell=p\), let \(D=B\otimes_{\bbQ} \bbQ_{p}\).
Then \(D\) is the quaternion division algebra over \(\bbQ_{p}\), and \(\calO_{p}\) a maximal order.

For an unramified prime \(\ell\), the quaternion algebra \(B\) splits at \(\ell\).
This splitting implies the following.

\begin{lemma}
\label{lemma:5.1}
Let \(\ell\neq p\) be unramified in \(B\), let \(\scrO_{F}=\bbZ_{\ell}\) and \(F=\bbQ_{\ell}\).
Then, there is an \(\scrO_{F}\)-algebra isomorphism \(\iota:B_{\ell}\stosim \rmop{M}_{2}(F)\) carrying quaternion conjugation to matrix conjugation.
More precisely, \(\iota(\calO_{\ell}) = \rmop{M}_{2}(\scrO_{F})\), and \(\iota(\bar{b})= J_{2}^{-1} \iota(b)^{T} J_{2}\), where \(J_{2} = \tm{0 & 1 \\ -1 & 0}\).
The induced blockwise isomorphism \(\rho:\rmop{M}_{2}(B_{\ell})\stosim \rmop{M}_{4}(F)\) then carries reduced trace to ordinary matrix trace, and conjugate transpose \(X\mapsto X^{*}\) to conjugation.
Let \(J_{0}=\tm{ J_{2} & 0 \\ 0 & J_{2}}\), \(H_{Y}=\rho(h_{Y})\), \(\Omega_{Y}=J_{0}H_{Y}\; (Y=A,C)\).
Then \(\rho(\scrT_{\ell})=\rmop{M}_{4}(\scrO_{F})\), \(\rmop{Trd}(X)=\rmop{tr}(\rho(X))\), and
\[
 \rho(X^{*})
 =J_{0}^{-1}\rho(X)^{T}J_{0},\quad
 \rho(X^{\dagger_{h_{Y}}})
 =\Omega_{Y}^{-1}\rho(X)^{T}\Omega_{Y}.
\]
Each \(\Omega_{Y}\) is in \(\rmop{GL}_{4}(\scrO_{F})\) and defines an alternating form on \(\scrO_{F}^{4}\) given by \(\langle v,w\rangle_{Y}=v^{T}\Omega_{Y}w\).
If \(P=\rho(\Psi)\), the multiplier equation becomes \( P^{T}\Omega_{A}P=M\Omega_{C}\).
\end{lemma}

We suppress \(\rho\) for clarity, and use the same letter for a quaternionic matrix in \(\rmop{M}_{2}(B)\) and its image in \(\rmop{M}_{4}(F)\).

Recall that \(\scrO_{F}=\bbZ_{\ell}\) and \(F=\bbQ_{\ell}\).
We recall some definitions, see \cite{Knus1991} and \cite{MilnorHusemoller1973}.
Let \(V\) be a free \(\scrO_{F}\)-module of rank \(2n\).
An \(\scrO_{F}\)-bilinear form \(\langle\;,\;\rangle\) on \(V\) is called \emph{alternating} if \(\langle v,v\rangle=0\) for every \(v\in V\).
It is \emph{nonsingular} if the determinant of its Gram matrix in an \(\scrO_{F}\)-basis is a unit.
Since \(\scrO_{F}\) is commutative, this is equivalent to the Gram matrix being invertible.
A \emph{symplectic basis} is an ordered \(\scrO_{F}\)-basis \((e_{1},\ldots,e_{n},f_{1},\ldots,f_{n})\) such that for \(1\leq i,j\leq n\),
\[
\langle e_{i},e_{j}\rangle = \langle f_{i},f_{j}\rangle =0,\quad
\langle e_{i},f_{j}\rangle = \delta_{ij}.
\]
For nonsingular forms, such bases over \(\scrO_{F}\) exist (\cite[Chap. I, Cor. (3.5)]{MilnorHusemoller1973}).

\begin{lemma}
\label{lemma:symplectic-normal-form}
Let \(F=\bbQ_{\ell}\) and \(\scrO_{F}=\bbZ_{\ell}\).
The uniformizer of \(\scrO_{F}\) is \(\pi=\ell\).
Let \(V_{C}\) and \(V_{A}\) be free \(\scrO_{F}\)-modules of rank \(4\), equipped with nonsingular alternating \(\scrO_{F}\)-bilinear forms \(b_{Y}=\langle\;,\;\rangle_{Y}\), \(Y=A,C\).
Assume given an \(\scrO_{F}\)-linear map \( P :V_{C} \to V_{A}\) such that
\[
\langle  P  v, P  w\rangle_{A} = \mu \, \langle v,w\rangle_{C},\quad \text{where}~
\mu=u\pi^{m},\; u\in \scrO_{F}^{\times},\; m\geq 0.
\]
Then, there exist symplectic bases for \(V_{C}\) and \(V_{A}\), ordered as \(e_{1},e_{2},f_{1},f_{2}\), in which the matrix of \(P\) is
\[
\rmop{diag}(\pi^{a}, \pi^{b}, u\pi^{m-a}, u\pi^{m-b}),\quad 0\leq a\leq b \leq \frac{m}{2}.
\]
\end{lemma}

\begin{proof}
We refer to Shimura's paper \cite[Proposition 1.6]{Shimura1963} for the proof.
\end{proof}

\begin{proposition}
\label{proposition:upper-bound-local-lattice-index-unramified}
Let \(\ell\neq p\) and put \(m=v_{\ell}(M)\).
Then, \([\scrS_{A,\ell}:L_{\psi,\ell}]\leq \ell^{4m}\).
\end{proposition}

\begin{proof}
Let \(\scrO_{F}=\bbZ_{\ell}\) with \(\ell\) unramified.
The prime \(\ell\)  is a uniformizer.
Write \(M=u\ell^{m}\), \(u\in \scrO_{F}^{\times}\).
Recall that \(\psi:C\to A\) is our initial prescribed polarized isogeny.
Its coordinate matrix is \(\Psi\in \rmop{M}_{2}(\calO)\), with scalar extension \(\Psi_{\ell}\) in \(\rmop{M}_{2}(\calO_{\ell})\).

Using Lemma \ref{lemma:5.1}, we transport coordinates to \(\rmop{M}_{4}(\scrO_{F})\).
We have an integral algebra isomorphism \(\rho:\rmop{M}_{2}(\calO_{\ell})\stosim \rmop{M}_{4}(\scrO_{F})\) and invertible matrices \(\Omega_{Y}\), \(Y=A,C\).
We put \(P=\rho(\Psi_{\ell})\).
By Lemma \ref{lemma:5.1},
\begin{align*}
\rho(\scrS_{A,\ell})
& = \scrS_{\Omega_{A}}
 = \{X\in \rmop{M}_{4}(\scrO_{F}) \mid \Omega_{A}X=X^{T}\Omega_{A},\, \tr(X)=0\},\\
\rho(L_{\psi,\ell})
& = \scrS_{\Omega_{A}} \cap P\, \rmop{M}_{4}(\scrO_{F}).
\end{align*}
Moreover, the multiplier equation is now \(P^{T}\Omega_{A}P=M\Omega_{C}\).
By Lemma \ref{lemma:symplectic-normal-form}, there exist matrices \(U_{A},U_{C}\in \rmop{GL}_{4}(\scrO_{F})\) such that
\[
U_{A}^{T} \Omega_{A} U_{A} = U_{C}^{T} \Omega_{C} U_{C} = J,\quad
J=\m{0 & I_{2} \\ -I_{2} & 0},
\]
and
\[
D=U_{A}^{-1}PU_{C} = \rmop{diag}(\ell^{a},\ell^{b},u\ell^{m-a},u\ell^{m-b}),\quad 0\leq a\leq b \leq \frac{m}{2}.
\]
Conjugation by \(U_{A}\) maps \(\scrS_{\Omega_{A}}\) to
\[
\scrS_{J} = \{X\in \rmop{M}_{4}(\scrO_{F}) \mid X^{T}J=JX,\, \tr(X)=0\}.
\]
Similarly, conjugation by \(U_{A}\) maps the right ideal \(P\rmop{M}_{4}(\scrO_{F})\) to \(D\, \rmop{M}_{4}(\scrO_{F})\).
We obtain an isomorphism of \(\scrO_{F}\)-modules
\[
\scrS_{A,\ell}/L_{\psi,\ell} \cong 
\scrS_{J} / (\scrS_{J} \cap D \rmop{M}_{4}(\scrO_{F}))
\]
induced by \(X\mapsto U_{A}^{-1} \rho(X) U_{A}\).

We now compute this index.
In fact, we shall compute the quotient directly.
Let \(X=\tm{U & V \\ W & Z}\in \rmop{M}_{2}(\rmop{M}_{2}(\scrO_{F}))=\rmop{M}_{4}(\scrO_{F})\) with size \(2\) blocks.
Then, the equation \(X^{T}J=JX\) holds if and only if
\[
Z=U^{T},\quad
V=-V^{T},\quad
W^{T}=-W.
\]
Thus, the matrices \(U,V,W,Z\) have the form
\[
U=\m{z & x \\ y & -z},\quad
V=\m{0 & d \\ -d & 0},\quad
W=\m{0 & e \\ -e & 0},\quad
Z=\m{z & y \\ x & -z},
\]
where \(x,y,z,d,e\in \scrO_{F}\).
We conclude that every element \(X\in\scrS_{J}\) may be written uniquely in the form
\begin{equation}
\label{eq:explicit-symplectic-standard-isomorphism}
X(z,x,y,d,e)=
\left[\begin{array}{cc|cc}
z & x & 0 & d \\
y & -z & -d & 0
\rule[-1.5ex]{0pt}{0pt} \\
  \hline
  \rule{0pt}{2ex}
0 & e & z & y\\
-e & 0 & x & -z
\end{array}\right],\quad
z,x,y,d,e\in \scrO_{F},
\end{equation}
with precisely 5 variables.
Conversely, any matrix of the form \(X(z,x,y,d,e)\) has trace zero and is self-adjoint with respect to \(J\), i.e., \(X^{T}J=JX\), whence \eqref{eq:explicit-symplectic-standard-isomorphism} is an explicit \(\scrO_{F}\)-linear isomorphism \(\scrO_{F}^{5}\stosim \scrS_{J}\).

Since \(D\) is invertible over \(F\), we have \(X\in D\rmop{M}_{4}(\scrO_{F})\) if and only if \(D^{-1}X\in \rmop{M}_{4}(\scrO_{F})\).
Left multiplication by \(D^{-1}\) divides each row by the corresponding entry of \(D\).
Since \(u\) is a unit, we see from \eqref{eq:explicit-symplectic-standard-isomorphism} that \(D^{-1}X\in \rmop{M}_{4}(\scrO_{F})\) if and only if we have the four conditions
\[
z,x,d\in \ell^{a}\scrO_{F},\quad
y,z,d\in \ell^{b}\scrO_{F},\quad
e,z,y\in \ell^{m-a}\scrO_{F},\quad
e,x,z\in \ell^{m-b}\scrO_{F}.
\]
From the inequalities \(a\leq b \leq m-b\leq m-a\), these hold if and only if we have
\[
z,y,e\in \ell^{m-a}\scrO_{F},\quad
x\in \ell^{m-b}\scrO_{F},\quad
d\in \ell^{b}\scrO_{F}.
\]
We thus obtain the quotient
\[
\scrS_{A,\ell}/L_{\psi,\ell} \cong (\scrO_{F}/\ell^{m-a}\scrO_{F})^{\oplus 3} \oplus (\scrO_{F}/\ell^{m-b}\scrO_{F}) \oplus (\scrO_{F}/\ell^{b}\scrO_{F}).
\]
Since \(\bbZ_{\ell}/\ell^{j}\bbZ_{\ell}\cong \bbZ/\ell^{j}\bbZ\), we see that \(\scrO_{F}/\ell^{j}\scrO_{F}\) has \(\ell^{j}\) elements.
Thus, the local index is \([\scrS_{A,\ell}:L_{\psi,\ell}] \leq \ell^{4m}\), completing the proof.
\end{proof}

We now work at the ramified prime \(p\).
Let
\[
\scrO_{F}=\bbZ_{p},\quad
F=\bbQ_{p},\quad
D=B\otimes_{\bbQ}F,\quad
\scrO_{D}=\calO \otimes_{\bbZ}\scrO_{F}.
\]
Thus, \(p\) is a uniformizer of \(\scrO_{F}\), \(D\) is the quaternion division algebra over \(F\), and \(\scrO_{D}\) is the unique maximal \(\scrO_{F}\)-order of \(D\), and also the valuation ring of \(D\),
\[
\scrO_{D}=\{x\in D \mid v_{D}(x)\geq 0\}.
\]

\begin{lemma}
\label{lemma:basic-quaternion-ramifed}
There exists an unramified quadratic extension \(K/F\) contained in \(D\), with ring of integers \(\scrO_{K}\), and a nontrivial automorphism \(\tau\in \rmop{Gal}(K/F)\), and an element \(\varpi\in D\), such that
\[
\scrO_{D} = \scrO_{K}\oplus \scrO_{K}\varpi,\quad
\varpi^{2}=p,\quad
\varpi a = \tau(a) \varpi \quad (a\in K).
\]
The maximal order \(\scrO_{D}\) contains a unique maximal ideal \(\frakP\), and this satisfies
\[
\frakP = p\scrO_{K} \oplus \scrO_{K}\varpi = \varpi \scrO_{D} = \scrO_{D} \varpi,\quad
\frakP^{2}=p\scrO_{D} = \scrO_{D} p.
\]
The valuation \(v_{p}\) extends to the valuation \(\frac{1}{2}v_{D}\) on \(D\), normalized by \(v_{D}(\varpi)=1\), and satisfies
\[
v_{D}(x)=v_{p}(\nrd(x))\quad (x\in D^{\times}),\quad
v_{D}(a) = 2v_{p}(a) \quad (a\in F^{\times}).
\]
We also put \(v_{D}(0)=\infty\).
Both maps \(\trd:\scrO_{D} \to \scrO_{F}\) and \(\nrd:\scrO_{D}^{\times} \to \scrO_{F}^{\times}\) are surjective.
Let \(\scrO_{D}^{0}=\ker(\trd:\scrO_{D} \to \scrO_{F})\).
Then, for any integer \(r\geq 0\), we have \(\frakP^{r} \cap F = p^{\lceil r/2\rceil} \scrO_{F}\) and \([\scrO_{D}^{0}:\scrO_{D}^{0}\cap \frakP^{r}] = p^{\lfloor 3r/2\rfloor}\).
\end{lemma}

\begin{proof}
Most of the above facts are classical, for instance, see \cite[\S 3.1]{Shimura1963}.
We have \(D=K\oplus K\varpi\).
To see that the reduced trace is surjective, we can argue as follows.
From the description of \(\scrO_{D}=\scrO_{K} \oplus \scrO_{K}\varpi\), let \(a,b\in \scrO_{K}\), and observe that \(\trd(a+b\varpi)=a+\tau(a)=\Tr_{K/F}(a)\).
Thus, \(\trd(\scrO_{D})=\Tr_{K/F}(\scrO_{K})\).
Composing \(\Tr_{K/F}\) with reduction modulo \(p\) coincides with \(\Tr_{\bbF_{p^{2}}/\bbF_{p}}\) in the sense that \(\Tr_{K/F}(a) \pmod{p} = \Tr_{\bbF_{p^{2}}/\bbF_{p}}(a\bmod{p})\) for \(a\in \scrO_{K}\) (see \cite[Chap. II, \S 9]{Neukirch1999}).
We can express the trace map \(\Tr_{\bbF_{p^{2}}/\bbF_{p}}(z)=z+z^{p}\).
This is \(\bbF_{p}\)-linear and is nonzero, so it is surjective since the target is one-dimensional.
Pick \(a\in \scrO_{K}\) whose residue mod \(p\) has trace \(1\), let \(t=\Tr_{K/F}(a)=\trd(a)\).
Then, \(\trd(a) \equiv 1 \pmod{p}\), so that \(t\in \bbZ_{p}^{\times}\).
Since \(t^{-1}\in F\) and \(\trd\) is \(F\)-linear, we get \(\trd(t^{-1}a)=t^{-1}\trd(a)=1\).
Since \(t^{-1}\in \bbZ_{p}\), we have \(t^{-1}a\in \scrO_{K}\).
This shows that \(\trd\) is surjective on \(\scrO_{D}\).

Recall that \(\frakP = \varpi\scrO_{D} = \scrO_{D} \varpi\), so the powers satisfy \(\frakP^{j}=\varpi^{j} \scrO_{D}=\scrO_{D}\varpi^{j}\).
The normalization of \(v_{D}\) gives \(v_{D}(\varpi)=1\).
Since \(\scrO_{D}\) is the valuation ring of \(D\), we have, for \(a\in F^{\times}\), \(a\in \frakP^{j} \cap F\) if and only if \(v_{F}(a) \geq \left\lceil \frac{j}{2}\right\rceil\).
Thus, \(\frakP^{r}\cap F = p^{\lceil r/2\rceil}\scrO_{F}\).

Let \(a,b\in \scrO_{K}\).
Since \(b\in K\), \(\bar{b}=\tau(b)\), and \(\bar{\varpi}=-\varpi\).
From \(\varpi b = \tau(b) \varpi\), we also get \(\overline{b\varpi}=-b\varpi\).
Hence, \(\trd(a+b\varpi)=\Tr_{K/F}(a)\).
Thus, \(\ker(\trd|_{\scrO_{D}})=\ker(\Tr_{K/F}|_{\scrO_{K}})\oplus \scrO_{K}\varpi\).
Write \(\scrO_{K}^{0}=\ker(\Tr_{K/F}:\scrO_{K} \to \bbZ_{p})\).
Then, \(\scrO_{D}^{0} = \scrO_{K}^{0} \oplus \scrO_{K} \varpi\).
Thus, \(\frakP^{r}=p^{\lceil r/2\rceil} \scrO_{K} \oplus p^{\lfloor r/2\rfloor}\scrO_{K}\varpi\).

Put \(s=\lceil r/2\rceil\) and \(t=\lfloor r/2\rfloor\).
Then,
\[
\scrO_{D}^{0}\cap \frakP^{r} = p^{s} \scrO_{K}^{0} \oplus p^{t} \scrO_{K}\varpi.
\]
Now, let \(e \in \scrO_{K}\) be an element with \(\Tr_{K/F}(e)=1\).
Then, every \(a\in \scrO_{K}\) decomposes as a sum of two terms, one of trace zero, and \(\Tr_{K/F}(a)e\).
This gives a direct sum decomposition \(\scrO_{K}=\scrO_{K}^{0}\oplus \scrO_{F}e\).
We conclude that the index is \(p^{s+2t}\), which completes the proof.
\end{proof}

\begin{lemma}
\label{lemma:quaternion-ramified-integral-change-of-basis}
Let \(h\in \rmop{GL}_{2}(\scrO_{D})\) be hermitian and let \(H=\tm{ 0 & 1 \\ 1 & 0}\).
\begin{enumerate}[autoroman]
\item There exists \(U\in \rmop{GL}_{2}(\scrO_{D})\) such that \(U^{*}hU=H\) and \(U^{-1}=HU^{*}h\).
\item If a vector \(e\in \scrO_{D}^{2}\) satisfies \(e^{*}he=0\) and has at least one unit coordinate in \(\scrO_{D}^{\times}\), then \(U\) in \textup{(i)} may be chosen with first column \(e\).
\end{enumerate}
\end{lemma}

\begin{proof}
Define the hermitian form \(\beta:D^{2}\times D^{2} \to D\) by \(\beta(v,w)=v^{*}hw\) for \(v,w\in D^{2}\).
The function \(Q:D^{2} \to D\) given by \(Q(v)=\beta(v,v)\) is actually \(F\)-valued.
By restriction of scalars, \(\dim_{F}D^{2}=8\).
Thus, the map \(Q:D^{2} \to F\) is a quadratic form over the field \(F\), and its (half-)polar form is nonsingular.
By \cite[Theorem 6.4.2, p.~217]{Scharlau1985}, it is isotropic.
Thus, there exists a vector \(w\in D^{2}\setminus\{0\}\) such that \(Q(w)=0=w^{*}hw\).
We fix this vector \(w\).
Now set \(n=\min \{v_{D}(w_{1}),v_{D}(w_{2})\}\) and \(e=w\varpi^{-n}\).
Then, \(e\in \scrO_{D}^{2}\) satisfies the hypotheses of (ii).

We now complete any vector \(e\) satisfying the conditions of (ii).
First, we observe that the row vector \(r=e^{*}h\) must contain a unit entry.
For otherwise, \(r=e^{*}h\in M_{1,2}(\frakP)\), but then, \(e^{*}=rh^{-1}\in \rmop{M}_{1,2}(\frakP)\) as well, since \(h^{-1}\in \rmop{M}_{2}(\scrO_{D})\) and \(\frakP\) is a two-sided ideal of \(\scrO_{D}\), contradicting the fact that one coordinate of \(e\) must be a unit.
Let \(r_{j}\) be a unit entry of \(r\), and let \(f_{0}\in \scrO_{D}^{2}\) be a vector with \(j\)-th coordinate \(r_{j}^{-1}\) and the other coordinate zero.
Then, \(\beta(e,f_{0})=\beta(f_{0},e)=1\).
We also have \(\beta(e,e)=0\).
Let \(t=\beta(f_{0},f_{0})\in F\cap \scrO_{D}=\scrO_{F}\).
Since \(\trd:\scrO_{D}\to \scrO_{F}\) is surjective by Lemma \ref{lemma:basic-quaternion-ramifed}, let \(c\in \scrO_{D}\) be such that \(\trd(c)=t=c+\bar{c}\).
Put \(f=f_{0}-ec\).
Then,
\[
\beta(e,f)=1,\quad
\beta(f,f) = 0.
\]
This shows that if \(U=(e\;f)\in \rmop{M}_{2}(\scrO_{D})\), then \(U^{*}hU=H\).
Since \(H^{2}=I_{2}\), we have \((HU^{*}h)(U)=H(U^{*}hU)=H^{2}=I_{2}\), so \(HU^{*}h\in \rmop{M}_{2}(\scrO_{D})\) is a left inverse of \(U\).
So \(U\) is bijective on \(D^{2}\) with inverse \(HU^{*}h\).
This proves (ii) and (i) together.
\end{proof}

We now diagonalize.

\begin{lemma}
\label{lemma:ramified-diagonalization}
Let \(\Theta\in \rmop{M}_{2}(\scrO_{D})\) satisfy \(\Theta^{*}H\Theta=MH\), where \(M=p^{m}u\), \(m\geq 0\), and \(u\in \scrO_{F}^{\times}\).
Let \(a=\min \{v_{D}(\Theta_{ij})\mid 1\leq i,j\leq 2\}\), with \(v_{D}(0)=\infty\).
Then, \(0\leq a\leq m\), and there exist matrices \(V_{A},V_{C}\in \rmop{GL}_{2}(\scrO_{D})\) such that
\[
V_{Y}^{*}HV_{Y}=H\quad (Y=A,C),\quad
V_{A}^{-1}\Theta V_{C} = \rmop{diag}(\varpi^{a}, M\bar{\varpi}{}^{-a}).
\]
\end{lemma}

\begin{proof}
Since \(H^{2}=I_{2}\), we have \(H(\Theta^{*}H\Theta)=MH^{2}=M\,I_{2}\), so the matrix \(M^{-1} H\Theta^{*}H\) is a left inverse of \(\Theta\).
Thus, \(\Theta\in \rmop{GL}_{2}(D)\).
The inequality \(a\geq 0\) is clear since \(\Theta\in \rmop{M}_{2}(\scrO_{D})\).
We will prove the upper bound below.

Observe that if \(P,Q\in \rmop{GL}_{2}(\scrO_{D})\), then, for every \(j\geq 0\), \(\Theta\in \rmop{M}_{2}(\frakP^{j}) \) if and only if \(P\Theta Q\in \rmop{M}_{2}(\frakP^{j})\).

Since right multiplication by \(H\) swaps columns, we can consider \(\Theta\) or \(\Theta H\), and one of these matrices will have a first column with an entry of minimum valuation \(a\) (either in the first column, or in the second column in which case we swap).
Thus let \(S\in\{I_{2},H\}\) be such that \(\Theta S\) has first column \(w\), with an entry of minimum valuation \(a\).
Either choice satisfies \(S^{*}HS=H\) since \(H^{2}=I_{2}\).
Since the diagonal entries of \((\Theta S)^{*}H(\Theta S)=MH\) are zero, we obtain \(w^{*}Hw=0\).
Thus, if we put \(v=w\varpi^{-a}\), then \(v^{*}Hv=0\), and \(v\) has a unit coordinate (of \(v_{D}\)-valuation zero) in \(\scrO_{D}^{\times}\).
By Lemma \ref{lemma:quaternion-ramified-integral-change-of-basis}(ii), we obtain a matrix \(U\in \rmop{GL}_{2}(\scrO_{D})\) with first column \(v\), such that \(U^{*}HU=H\).
Thus, we can let \(\Theta_{1}=U^{-1}\Theta S\) and write
\[
\Theta_{1}= \m{ \varpi^{a} & b \\ 0 & d} \in \rmop{M}_{2}(\scrO_{D}),\quad
\Theta_{1}^{*} H \Theta_{1} = MH,
\]
and the minimum entry valuation is still \(a\).
Thus the equation \(\Theta_{1}^{*}H\Theta_{1}=MH\) yields the equations
\[
\bar{d}\varpi^{a}=M,\quad
\bar{b}d+\bar{d}b=0.
\]
We get \(d=M\bar{\varpi}{}^{-a}\) and so \(v_{D}(d)=2m-a\geq a\), hence \(a\leq m\).
Let \(c=\varpi^{-a}b\in \scrO_{D}\).
The second equation then implies \(\trd(c)=c+\bar{c}=0\).
Finally, let \(T=\tm{1 & -c \\ 0 & 1}\).
Then, we have \(T^{*}HT=H\) and \(\Theta_{1}T=\rmop{diag}(\varpi^{a},M\bar{\varpi}{}^{-a})\).
Taking \(V_{A}=U\) and \(V_{C}=ST\) completes the proof.
\end{proof}

\begin{proposition}
\label{proposition:upper-bound-local-lattice-index-ramified}
Let \(m=v_{p}(M)\).
Then, \([\scrS_{A,p}:L_{\psi,p}]\leq p^{4m}\).
\end{proposition}

\begin{proof}
We apply Lemma \ref{lemma:quaternion-ramified-integral-change-of-basis}.
Since \(h_{Y}=h_{Y}^{*}\in \rmop{GL}_{2}(\scrO_{D})\) for \(Y=A,C\), we get matrices \(W_{Y}\in \rmop{GL}_{2}(\scrO_{D})\) such that \(W_{Y}^{*}h_{Y}W_{Y}=H\).
Now put \(\Theta=W_{A}^{-1} \Psi W_{C}\in \rmop{M}_{2}(\scrO_{D})\).
Then,
\[
\Theta^{*}H\Theta=MH.
\]
Let \(a\) be the minimum \(v_{D}\)-valuation of the entries of \( \Psi \).
The minimum entry valuation of \(\Theta\) is \(a\), by the observation in the proof of Lemma \ref{lemma:ramified-diagonalization}.
Applying the lemma, we obtain matrices \(V_{A},V_{C}\in \rmop{GL}_{2}(\scrO_{D})\) such that
\[
V_{Y}^{*}HV_{Y}=H\quad (Y=A,C),\quad\text{and}\quad
V_{A}^{-1}\Theta V_{C} = \m{ \varpi^{a} \\ & M\bar{\varpi}{}^{-a}}.
\]
For \(Y=A,C\), we now define \(U_{Y}=W_{Y}V_{Y}\), and put \(\Delta=U_{A}^{-1} \Psi U_{C}\), which is diagonal.
By Lemma \ref{lemma:pairs-local-lattices}, we have
\[
\scrS_{A,p}/L_{\psi,p} \cong \scrS_{H,p} / \bigl(\scrS_{H,p}\cap \Delta \rmop{M}_{2}(\scrO_{D})\bigr),
\]
where
\[
\scrS_{H,p} = \{X\in \rmop{M}_{2}(\scrO_{D}) \mid X^{*}H=HX,\, \Trd(X)=0\}.
\]

Write \(X=\tm{z & b \\ c & w}\).
Thus, \(X^{*}H=HX\) if and only if \(b,c\in F\cap \scrO_{D}=\scrO_{F}\) and \(w=\bar{z}\).
Thus, \(X=\tm{z & b \\ c & \bar{z}}\), \(b,c\in \scrO_{F}\).
Then, \(\Trd(X)=\trd(z)+\trd(\bar{z})=2\trd(z)\).
Since multiplication by \(2\) is injective on \(\scrO_{F}\), \(\Trd(X)=0\) if and only if \(\trd(z)=0\).
Thus, we obtain an explicit representation of \(\scrS_{H,p}\) as
\begin{equation}
\label{eq:diagonalized-lattice-ramified}
\scrS_{H,p} = \left\{\m{z & b \\ c & -z} \;\middle|\;z\in \scrO_{D}^{0},\, b,c\in \scrO_{F}\right\}.
\end{equation}

Put \(r=2m-a\geq a\).
Now we write \(M=p^{m}u\) with \(u\in \scrO_{F}^{\times}\).
Then, \(M\bar{\varpi}{}^{-a}=(-1)^{a}u\varpi^{r}\).
Therefore, the diagonal entries of \(\Delta\) generate the right ideals \(\frakP^{a}\) and \(\frakP^{r}\).
So for a matrix \(X=\tm{z & b \\ c & -z}\in \scrS_{H,p}\cap \Delta \rmop{M}_{2}(\scrO_{D})\) as in \eqref{eq:diagonalized-lattice-ramified}, this imposes the conditions \(b\in \frakP^{a}\) and \(c,z\in \frakP^{r}\).
By Lemma \ref{lemma:basic-quaternion-ramifed}, \(\frakP^{r}\cap F=p^{\lceil r/2\rceil} \scrO_{F}\), so we obtain the three conditions
\[
z\in \scrO_{D}^{0}\cap \frakP^{r},\quad
b\in p^{\lceil a/2\rceil}\scrO_{F},\quad
c\in p^{\lceil r/2\rceil}\scrO_{F}.
\]
Thus, we have identified the quotient \(\scrS_{H,p}/(\scrS_{H,p}\cap \Delta \rmop{M}_{2}(\scrO_{D}))\) with
\[
(\scrO_{D}^{0}/\scrO_{D}^{0}\cap \frakP^{r})\oplus
(\scrO_{F}/p^{\lceil a/2\rceil}\scrO_{F})\oplus
(\scrO_{F}/p^{\lceil r/2\rceil}\scrO_{F}).
\]
The index is therefore \(p^{\lfloor 3r/2\rfloor+\lceil a/2\rceil+\lceil r/2\rceil}\leq p^{4m}\) by Lemma \ref{lemma:basic-quaternion-ramifed}.
This completes the proof.
\end{proof}

We can now prove Proposition \ref{proposition:index-bound}.

\begin{proof}[Proof of Proposition \ref{proposition:index-bound}]
Let \(J=[\scrS_{A}:L_{\psi}]\).
The local lattices in Propositions \ref{proposition:upper-bound-local-lattice-index-unramified} and \ref{proposition:upper-bound-local-lattice-index-ramified} are the scalar extensions of this global pair, and \(J\) is the product of these local indices.
Since the local index is \(1\) whenever \(\ell\nmid M\), we get
\[
J=\prod_{\ell} [\scrS_{A,\ell}:L_{\psi,\ell}]\leq \prod_{\ell\mid M} \ell^{4v_{\ell}(M)} = M^{4}.
\]
This implies the first inequality in \eqref{eq:lattice-index-global-upper-bound}.

Finally, equation \eqref{eq:finite-index-sublattice-composition} from Proposition \ref{proposition:factorization-sublattice} gives
\[
\det(L_{\psi},q_{\psi}) = \frac{p^{2}}{16M^{5}} [\scrS_{A}:L_{\psi}]^{2}
\leq\frac{p^{2}M^{8}}{16M^{5}}=\frac{p^{2}M^{3}}{16},
\]
proving the second inequality in \eqref{eq:lattice-index-global-upper-bound}.
\end{proof}

\section{Shortening and the uniform bound}

We now proceed with the same Hermite inequality and isogeny reconstruction.
The crucial new ingredient is that the determinant now depends on the multiplier \(M\) of a reference prescribed isogeny.
Applying the estimate to an isogeny of \emph{least} multiplier removes this dependence.

For a positive definite quadratic lattice \((L,q)\) of rank \(n\), Hermite's inequality (with our determinant convention) states that
\begin{equation}
\label{eq:Hermite}
\min_{v\in L\setminus\{0\}} q(v) \leq \gamma_{n} \, \det(L,q)^{1/n},
\end{equation}
where \(\gamma_{n}\) denotes \emph{Hermite's constant}.
(See \cite[Chap. IX, \S 7]{Cassels1959}.)
The classical value in dimension five is \(\gamma_{5}=2^{3/5}\).

\begin{proposition}
\label{proposition:interlacing-multiplier-inequality}
Let \(\psi:C \to A\) be a polarized \(k\)-isogeny with integer multiplier \(M>0\), and let \(J=[\scrS_{A}:L_{\psi}]\).
Then, there exists a polarized \(k\)-isogeny \(f:A \to C\) with integer multiplier \(N\) such that
\begin{equation}
\label{eq:interlacing-multiplier-inequality}
1\leq N\leq \left(\frac{p^{2}J^{2}}{2M^{5}}\right)^{1/5} \leq \left(\frac{p^{2}M^{3}}{2}\right)^{1/5}.
\end{equation}
If \(M>p/\sqrt{2}\), then \(N<M\).
\end{proposition}

\begin{proof}
Apply Hermite's inequality \eqref{eq:Hermite} to \((L_{\psi},q_{\psi})\) as in Proposition \ref{proposition:factorization-sublattice}.
Applying Hermite's inequality to \eqref{eq:finite-index-sublattice-composition}, we obtain the existence of a nonzero vector \(x\in L_{\psi}\) such that
\[
q_{\psi}(x)\leq 2^{3/5}\det(L_{\psi},q_{\psi})^{1/5}=\left(\frac{p^{2}J^{2}}{2M^{5}}\right)^{1/5}\leq \left(\frac{p^{2}M^{3}}{2}\right)^{1/5}.
\]
This proves \eqref{eq:interlacing-multiplier-inequality}.
Put \(N=q_{\psi}(x)\).
By Proposition \ref{proposition:factorization-sublattice}, \(x\) has a unique expression \(x=\psi\circ f\) with \(f\in \Hom_{k}(A,C)\), and where \(f\) satisfies \(f^{\dagger}f=[N]_{A}\).
Thus, \(f\) is the desired isogeny.
If \(p/\sqrt{2}<M\), the inequality \eqref{eq:interlacing-multiplier-inequality} implies \(N<M\).
\end{proof}

\begin{remark}
This construction actually recovers the Frobenius-conjugate bound of \cite{Pham2026}.
For this, we explicitly choose \(E=E_{0}\times_{\bbF_{p}} k\) with \(E_{0}\) over \(\bbF_{p}\) with \(p\)-Frobenius \(\pi\), such that \(\pi^{2}=-p\) (see \cite{JordanKeetonPoonenRainsShepherdBarronTate2018}).
Then, \(E^{\sigma}\) is canonically \(E\), and \(\pi\calO=\calO \pi=\frakP\).
Let \(\sigma\) be the nontrivial automorphism of \(k/\bbF_{p}\) and write \(F_{Y}:Y \to Y^{\sigma}\) for relative \(p\)-Frobenius.
In the current construction, take \(C=A^{\sigma}\) with its conjugate polarization, and let \(\psi=F_{A^{\sigma}}:A^{\sigma}\to A\), with multiplier \(p\).
Choosing \(\iota_{C}=\iota_{A}^{\sigma}\), the matrix of \(\psi\) is \(\pi I_{2}\), and in these coordinates, we have the lattice \(L_{\psi}=\scrS_{h_{A}}\cap \rmop{M}_{2}(\frakP)\), and \(q_{\psi}=q_{0}/p\).
This is the quadratic lattice used in \cite{Pham2026}, whose determinant is shown to be \(p^{3}/16\).
Substituting this determinant in the Hermite inequality argument in the proof of Proposition \ref{proposition:interlacing-multiplier-inequality} then gives
\[
N \leq \gamma_{5}(p^{3}/16)^{1/5}=(p^{3}/2)^{1/5}<p,
\]
and hence a separable polarized isogeny \(A\to A^{\sigma}\) with the desired bound.
\end{remark}

We can now complete the proof of Theorem \ref{theorem:A}.

\begin{proof}[Proof of Theorem \ref{theorem:A}]
By Proposition \ref{proposition:existence-multipliers}, we may choose a polarized isogeny \(\psi:C\to A\) with least positive integer multiplier \(M\).
If \(M>p/\sqrt{2}\), Proposition \ref{proposition:interlacing-multiplier-inequality} gives a polarized isogeny \(f:A \to C\) of multiplier \(N<M\).
Its polarized adjoint \(f^{\dagger}:C \to A\) is also an isogeny with the same multiplier \(N\), which contradicts minimality.
Thus, \(M\leq \lfloor p/\sqrt{2}\rfloor\).
Since \(M<p\), the degree \(M^{2}\) is prime to \(p\), so the isogeny \(\psi^{\dagger}\) is separable.
We already explained the descent from \(\bar{k}\) to \(k\), and this isogeny can be transported back to \(\bar{k}\).
\end{proof}

\begin{proposition}
\label{proposition:lowerbound}
For every prime \(p>12\) with \(p\equiv 3 \pmod{4}\), there are principally polarized superspecial abelian surfaces over \(\bar{k}\) such that every polarized isogeny between them has multiplier at least \(p/12\).
\end{proposition}

\begin{proof}
Let \(R\) be a maximal order in \(B_{p,\infty}\).
Let \(x,y\in R\), and assume that they do not commute: \(xy\neq yx\).
Then, \(K=\bbQ(x)\) is a quadratic field extension and \(y\notin K\).
Then, \(1,x,y,xy\) are linearly independent over \(\bbQ\), since a nontrivial relation \(a+bx+(c+dx)y=0\) with \(c+dx\neq 0\) would imply \(y\in K\).
Thus, \(L=\bbZ +\bbZ x + \bbZ y + \bbZ xy\subset R\) is a finite-index sublattice.
Using the inner product
\[
\langle u,v\rangle = \trd(u\bar{v}),\quad \|u\|^{2}=2\nrd(u),
\]
we have \(\rmop{covol}(R)=p\) (see \cite[Theorem 15.5.5]{Voight2021}).
Hadamard's inequality (see \cite[Chap. V, Lem. 9]{Cassels1959}) gives
\[
p = 1\cdot p \leq [R:L]\cdot p = \rmop{covol}(L)
\leq \|1\|\cdot\|x\|\cdot\|y\|\cdot\|xy\|
=4\nrd(x)\nrd(y).
\]
We conclude that if \(\nrd(x)\nrd(y)<p/4\), then \(xy=yx\).
This argument is a variant of \cite[Cor. 2.1.2]{GorenLauter2007}.

This has the following consequence:
let \(p>3\) be a prime with \(p\equiv 3\pmod{4}\), and let \(E\) be a supersingular elliptic curve over \(\bar{\bbF}_{p}\) with \(j(E)=1728\).
Let \(\calO=\End_{\bar{\bbF}_{p}}(E)\).
There is a quadratic suborder of \(\calO\), \(R=\bbZ[\iota]\), where \(\iota\) is an automorphism of \(E\) satisfying \(\iota^{2}=[-1]\).
Since \(\deg(\iota)=\nrd(\iota)=1\), it follows that if \(\alpha\in \calO\) is an endomorphism with \(\deg(\alpha)<p/4\), then \(\alpha\iota=\iota\alpha\).
Since \(\iota\) is not central, its centralizer \(C_{B}(\iota)\) is \(K=\bbQ(\iota)\), hence \(\alpha\in \bbQ(\iota)\cap \calO=\bbZ[\iota]\).

Assume \(p>12\).
Let \(E\) be as above and let \(f:E\to F\) be an isogeny of degree \(3\).
Let \(Q=E^{2}\) and \(P=E\times F\) equipped with their product principal polarizations.
If there is an isomorphism \(\varphi:F\stosim E\), then \(\varphi\circ f\in \End(E)\) is an endomorphism of degree \(3<p/4\).
By the above remarks, this endomorphism is a Gaussian integer of (reduced) norm \(3\), which is not possible, hence \(Q\) and \(P\) are not isomorphic as principally polarized abelian varieties.
Let \(T:Q \to P\) be an isogeny with multiplier \(N<p/12\).
Let \(D=1_{E}\times f^{\dagger}:P \to Q\).
Then, \(u=D\circ T\in \End(E^{2})\) has degree \(3N^{2}\), since
\[
uu^{\dagger}=D(TT^{\dagger})D^{\dagger}
=D[N]_{P}D^{\dagger}
=[N]_{Q} \circ (1_{E} \times f^{\dagger}f)=[N]_{E}\times [3N]_{E}.
\]
If \(U=(u_{ij})\in \rmop{M}_{2}(\calO)\) represents \(u\), this becomes
\[
UU^{*}=\m{N & 0 \\ 0 & 3N}.
\]
Hence, \(\nrd(u_{11})+\nrd(u_{12})=N\) and \(\nrd(u_{21})+\nrd(u_{22})=3N\).
Since both \(N,3N<p/4\), every entry has reduced norm \(<p/4\), so each entry belongs to \(\bbZ[\iota]\).
Then,
\[
\rmop{N}_{K/\bbQ}(\det U)=3N^{2}.
\]
But this is impossible since the valuation of \(3\) must be even.
Thus, every polarized isogeny \(E\times E \to E\times F\) must have multiplier \(\geq p/12\).
\end{proof}

\section{A deterministic quaternionic algorithm for small polarized isogenies}
\label{section:algorithm}

The input data consists of \(p\), a \(\bbZ\)-basis of \(\calO\) with multiplication and conjugation tables, and the coordinates of positive definite hermitian matrices \(h_{A},h_{C}\in \rmop{M}_{2}(\calO)\) of hermitian determinant \(1\).

\subsection*{Description of the algorithm}

Given two hermitian positive definite integral matrices \(h_{A},h_{C}\in \rmop{M}_{2}(\calO)\), with hermitian determinant \(1\), we consider pairs \((\gamma,N)\) with \(\gamma\in \rmop{M}_{2}(\calO)\) and \(N\in \bbZ_{>0}\) such that
\begin{equation}
\label{eq:original}
\gamma^{*}h_{C}\gamma = N h_{A}.
\end{equation}
The goal is to find such a pair with \(N\leq\lfloor p/\sqrt{2}\rfloor\).
This determines the isogeny \(A \to C\).

To find it, the algorithm starts in the other direction, namely, it starts with a matrix \(\Psi\) and an integer \(M>0\) such that
\begin{equation}
\label{eq:pair-equation}
\Psi^{*}h_{A}\Psi = M h_{C},
\end{equation}
that is, from an auxiliary isogeny \(C \to A\).
The algorithm successively finds pairs \((\Psi,M)\) satisfying this equation.
The polarized adjoint of \(\Psi\) is \(\Psi^{\dagger}=h_{C}^{-1}\Psi^{*}h_{A}\).
We fix the following quadratic form on \(\rmop{M}_{2}(B)\):
\begin{equation}
Q(Y)=\frac{1}{4} \Trd(h_{A}^{-1}Y^{*}h_{C}Y),\quad Y\in \rmop{M}_{2}(B).
\end{equation}

\subsubsection*{Step 1: Initialization}

We first find an arbitrary matrix \(\Psi\in \rmop{M}_{2}(\calO)\) solving the equation with some multiplier \(M\).

\begin{lemma}
Let
\[
\Psi = I_{2} + h_{A}^{-1} h_{C},\quad
M=\Hdet(h_{A}+h_{C}).
\]
Then, the pair \((\Psi,M)\) satisfies \eqref{eq:pair-equation}.
\end{lemma}

This is the content of Proposition \ref{proposition:existence-multipliers}.

\subsubsection*{Step 2: Shortening step}

Let \((\Psi,M)\) be a pair satisfying \eqref{eq:pair-equation}.
While \(2M^{2}>p^{2}\), form the lattice
\begin{equation}
\scrK_{\Psi} = \{Y\in \rmop{M}_{2}(\calO)\mid (\Psi Y)^{*} h_{A} = h_{A} (\Psi Y),\, \Trd(\Psi Y)=0\}.
\end{equation}
Observe that multiplication by \(\Psi\) identifies \(\scrK_{\Psi}\) with the matrix realization of \(L_{\psi}\).
If \(X=\Psi Y\), then \(Y \in \scrK_{\Psi}\) if and only if \(\Trd(X)=0\) and \(X^{*}h_{A}=h_{A}X\), that is, \(Y\in \scrK_{\Psi}\) if and only if \(X\in \scrS_{h_{A}}\), that is, \(\scrS_{h_{A}}\cap \Psi \rmop{M}_{2}(\calO)=\Psi\,\scrK_{\Psi}\).
By Proposition \ref{prop:basic-quadratic-form}, \(X^{2}=q_{0}(X)I_{2}\) where \(q_{0}(X)=\frac{1}{4}\Trd(X^{2})\).

Now, from \(\Psi^{*}h_{A}\Psi=M\, h_{C}\), we obtain
\[
Y^{*}(Mh_{C})Y = Y^{*}\Psi^{*}h_{A}\Psi Y = X^{*}h_{A}X = h_{A}X^{2} = q_{0}(X) h_{A}.
\]
Thus,
\[
Y^{*}h_{C}Y = \frac{q_{0}(X)}{M} h_{A},\quad\text{with}~
Q(Y)=\frac{1}{4} \Trd(h_{A}^{-1}Y^{*}h_{C}Y)=\frac{q_{0}(X)}{M}.
\]
Choose \(Y\in \scrK_{\Psi}\) nonzero minimizing \(Q(Y)\).
If we put \(N^{\prime}=Q(Y)\), we obtain
\begin{equation}
Y^{*}h_{C}Y=N^{\prime} h_{A},\quad N^{\prime} = Q(Y).
\end{equation}
Thus, the pair \((Y,N^{\prime})=(Y,Q(Y))\) satisfies \eqref{eq:original}.

Now let
\begin{equation}
\label{eq:final-bootstrap}
\Psi^{\prime} = h_{A}^{-1} Y^{*} h_{C},\quad (\Psi^{\prime})^{*}=h_{C}Yh_{A}^{-1}.
\end{equation}
Then, \(\Psi^{\prime} Y = N^{\prime} I_{2}\).
Since \(N^{\prime}>0\), \(Y\) is invertible with inverse \(\Psi^{\prime}/N^{\prime}\), hence also \(Y\Psi^{\prime} = N^{\prime} I_{2}\).
Then,
\[
(\Psi^{\prime})^{*} h_{A} \Psi^{\prime} = h_{C} Y \Psi^{\prime} = N^{\prime} h_{C}.
\]
Thus, the pair \((\Psi^{\prime},N^{\prime})\) satisfies \eqref{eq:pair-equation}, that is, the same as the starting pair \((\Psi,M)\).
By Proposition \ref{proposition:interlacing-multiplier-inequality},
\begin{equation}
\label{eq:algorithm-interlacing-step}
1\leq N^{\prime} \leq \left(\frac{p^{2}M^{3}}{2}\right)^{1/5}<M.
\end{equation}
Thus, we conclude:

\begin{proposition}[Shortening]
If \(Y\in \scrK_{\Psi}\) is a shortest nonzero vector for \(Q\), then the pair \((\Psi^{\prime},N^{\prime})\) with \(\Psi^{\prime}=h_{A}^{-1}Y^{*}h_{C}\) and \(N^{\prime}=Q(Y)\) satisfies \eqref{eq:pair-equation}.
If \(M>p/\sqrt{2}\), then \(N^{\prime}<M\).
\end{proposition}

Repeat this step while \(M>p/\sqrt{2}\).

\subsubsection*{Step 3: Output}

When \(M\leq p/\sqrt{2}\), we return the pair \((\gamma,N)\), where
\[
\gamma = h_{C}^{-1}\Psi^{*}h_{A},\quad N=M.
\]
Then, the pair \((\gamma,N)\) satisfies the original equation \eqref{eq:original}.

Referring to the above analysis, if the last returned pair is \((\Psi^{\prime},N^{\prime})\) as above, then that pair satisfies \eqref{eq:pair-equation}.
By \eqref{eq:final-bootstrap}, \(Y=h_{C}^{-1}(\Psi^{\prime})^{*}h_{A}\), and the pair \((Y,N^{\prime})\) satisfies the original equation \eqref{eq:original}.

\begin{proposition}
\label{proposition:shortening-steps}
The algorithm terminates and returns an integral matrix \(\gamma\) with
\[
\gamma^{*}h_{C}\gamma = N h_{A},\quad 1\leq N \leq \left\lfloor \frac{p}{\sqrt{2}}\right\rfloor<p.
\]
The number of shortening steps is at most \(1+2\log (4p^{2}\log M_{0})\).
\end{proposition}

\begin{proof}
Let \(P=p/\sqrt{2}\).
If \(M_{0}\leq P\), there is nothing to do.
Otherwise, \(M_{0}>P\) so let \(D=D_{0}=\log(M_{0}/P)>0\) and \(D_{t}=\log(M_{t}/P)\).
As long as the iterates are defined, \eqref{eq:algorithm-interlacing-step} shows that \(M_{t+1}^{5} \leq P^{2}M_{t}^{3}\).
Thus, for any \(t\geq 0\),
\[
\log \left(\frac{M_{t}}{P}\right) \leq \left(\frac{3}{5}\right)^{t} \log \left(\frac{M_{0}}{P}\right)=\left(\frac{3}{5}\right) ^{t}D_{0}.
\]
Thus,
\[
D_{t} \leq \left(\frac{3}{5}\right)^{t} D_{0}.
\]
Since \(M_{t}\) is an integer, if \(M_{t}>P\), we have \(2M_{t}^{2}-p^{2}\geq 1\), so \(2M_{t}^{2}\geq 1+p^{2}\).
Then, \(M_{t}^{2}/P^{2}=2M_{t}^{2}/p^{2} \geq 1+p^{-2}\).
Hence,
\[
D_{t}=\log(M_{t}/P)\geq \frac{1}{2} \log(1+p^{-2})\geq \frac{1}{4p^{2}}.
\]
Therefore, the loop stops once \((3/5)^{t}D<1/(4p^{2})\).
In other words, when \(4p^{2}D<(5/3)^{t}\), hence as soon as
\[
t>\frac{\log(4p^{2}D)}{\log(5/3)}.
\]
This proves the proposition.
\end{proof}

\subsection*{Bit complexity}

The algorithm runs in deterministic polynomial time in the total binary input length.
Computing a basis of \(\scrK_{\Psi}\) and computing an exact shortest vector in fixed rank are polynomial-time operations, while the initial upper bound for the quadratic form \(Q\) ensures that all matrices involved in the algorithm have polynomially bounded bit length.
The claim then follows from Proposition \ref{proposition:shortening-steps}.

\printbibliography

\end{document}